\documentclass[12pt]{article}
\usepackage{amsmath,amssymb,amscd}
\usepackage{kotex}
\usepackage[usenames]{color}
\allowdisplaybreaks

 \newtheorem{theorem}{Theorem}[section]
 \newtheorem{proposition}[theorem]{Proposition}
 \newtheorem{corollary}[theorem]{Corollary}
 \newtheorem{lemma}[theorem]{Lemma}
 
 \newtheorem{example}[theorem]{Example}
 \newtheorem{remark}[theorem]{Remark}

\numberwithin{equation}{section}

\newenvironment{proof}{\smallskip\par{\sc Proof.}\enspace}%
 {{\unskip\nobreak\hfil\penalty50\hskip2em
          \hbox{}\nobreak\hfil{\rule[-1pt]{5pt}{10pt}}
          \parfillskip=0pt\finalhyphendemerits=0
          \par\medskip}} 

\newenvironment{proof of main thm}{\smallskip\par{\sc \textbf{Proof of Theorem 3.1.}}\enspace}%
 {{\unskip\nobreak\hfil\penalty50\hskip2em
          \hbox{}\nobreak\hfil{\rule[-1pt]{5pt}{10pt}}
          \parfillskip=0pt\finalhyphendemerits=0
          \par\medskip}} 

\begin{document}

\vspace*{.3in}

\begin{center}
\LARGE
{\sf Littlewood-Paley Type Inequality for Evolution Systems Associated with Pseudo-Differential Operators}
\end{center}

\bigskip

\begin{center}
Un Cig Ji \\
Department of Mathematics\\
Institute for Industrial and Applied Mathematics\\
Chungbuk National University\\
Cheongju 28644, Korea \\
\texttt{E-Mail:uncigji@chungbuk.ac.kr }
\end{center}

\begin{center}
Jae Hun Kim \\
Department of Mathematics\\
Chungbuk National University\\
Cheongju 28644, Korea \\
\texttt{E-Mail:jaehunkim@chungbuk.ac.kr }
\end{center}

\begin{abstract}
In this paper, we first prove that the kernel of convolution operator,
corresponding the composition of pseudo-differential operator and evolution system associated with the symbol depending on time,
satisfies the H\"ormander's condition.
Secondly, we prove that the convolution operator is a bounded linear operator
from the Besov space on $\mathbb{R}^{d}$ into $L^{q}(\mathbb{R}^{d};V)$ for a Banach space $V$.
Finally, by applying the Calder\'{o}n-Zygmund theorem for vector-valued functions,
we prove the Littlewood-Paley type inequality for evolution systems associated with pseudo-differential operators.
\end{abstract}


Calder\'{o}n-Zygmund theorem, Littlewood-Paley inequality
\bigskip
\noindent

\newpage
\tableofcontents

\section{Introduction}
The Littlewood-Paley theory plays an important role in harmonic analysis, partial differential equations, and some related fields.
Especially, the Littlewood-Paley theory is a useful tool to characterize function spaces via the $L^{p}$-norm estimates of functions.
The Poisson semigroup $\{P_{t}\}_{t\geq 0}$ on $L^{p}(\mathbb{R}^{d})$ (for $1\le p\le \infty$) is defined by
\begin{equation}\label{eqn: Poisson semigp and kernel}
P_{t}f(x)=\left(k(t,\cdot)*f\right)(x),\quad k(t,x)=\frac{\Gamma\left(\frac{d+1}{2}\right)}{\pi^{\frac{d+1}{2}}}\frac{t}{(t^{2}+|x|^{2})^{\frac{d+1}{2}}}
\end{equation}
for $t>0$ and $f\in L^{p}(\mathbb{R}^{d})$, and $P_{0}=I$ (the identity operator),
where $*$ is the convolution in $\mathbb{R}^{d}$ and $k$ is the Poisson kernel on $\mathbb{R}^{d}$ (see \cite{Stein 1970}).
The Littlewood-Paley $g$-function for the Poisson semigroup $\{P_{t}\}_{t\geq 0}$ is defined by
\begin{equation}\label{eqn: g function for Poisson}
g(f)(x)=\left(\int_{0}^{\infty}\left|t\frac{\partial}{\partial t}P_{t}f(x)\right|^2 \frac{dt}{t}\right)^{\frac{1}{2}},
\end{equation}
and then the Littlewood-Paley inequality for the Poisson semigroup $\{P_{t}\}_{t\geq 0}$ is stated as follows.
\begin{theorem}[\cite{Stein 1970}]
For any $1<p<\infty$, there exist positive constants $C_{1}, C_{2}$ depending on $p$ and $d$ such that
\begin{equation}\label{eqn: classical LP}
C_{1}\|f\|_{L^{p}(\mathbb{R}^{d})}\leq \|g(f)\|_{L^{p}(\mathbb{R}^{d})}\leq C_{2}\|f\|_{L^{p}(\mathbb{R}^{d})}.
\end{equation}
\end{theorem}

There are several modified forms of the Littlewood-Paley $g$-function as given in \eqref{eqn: g function for Poisson}.
For instance, the operator $\frac{\partial}{\partial t}$ can be replaced by
$\nabla_{n+1}=\left(\frac{\partial}{\partial t},\frac{\partial}{\partial x_{1}},\cdots,\frac{\partial}{\partial x_{d}}\right)$ or
$\nabla=\left(\frac{\partial}{\partial x_{1}},\cdots,\frac{\partial}{\partial x_{d}}\right)$.
Additionally, the Poisson semigroup $\{P_{t}\}_{t\geq 0}$ in \eqref{eqn: g function for Poisson}
can be replaced by the heat semigroup $\{T_{t}\}_{t\geq0}$ defined by
\[
T_{t}f(x)=p(t,\cdot)*f(x),\quad p(t,x)=\frac{1}{(4\pi t)^{\frac{d}{2}}}e^{-\frac{|x|^2}{4t}}
\]
for $t>0$ and $f\in L^{p}(\mathbb{R}^{d})$, and $T_{0}=I$.
For more information, we refer to \cite{Stein 1970} and \cite{Frazier 1991}.

For a measure space $(X,\mu)$, a symmetric diffusion semigroup $\boldsymbol{\mathcal{T}}=\{\mathcal{T}_{t}\}_{t\ge0}$
is a family of operators defined on $\bigcup_{1\leq p\leq \infty}L^{p}(X,\mu)$ satisfying the following properties:
\begin{itemize}
  \item [(i)]$\mathcal{T}_{0}=1$ and $\mathcal{T}_{t+s}=\mathcal{T}_{t}\mathcal{T}_{s}$ for any $t,s>0$,
  \item [(ii)]$\lim_{t\to 0}\mathcal{T}_{t}f=f$ in $L^{2}(X,\mu)$ for any $f\in L^{2}(X,\mu)$,
  \item [(iii)]$\|\mathcal{T}_{t}f\|_{L^{p}(X,\mu)}\leq \|f\|_{L^{p}(X,\mu)}$ for any $1\leq p\leq \infty$ and $t>0$,
  \item [(iv)]$\mathcal{T}_{t}$ is selfadjoint on $L^{2}(X,\mu)$,
  \item [(v)] For any $f\geq 0$, it holds that $\mathcal{T}_{t}f\geq 0$,
  \item [(vi)] $\mathcal{T}_{t}1=1$.
\end{itemize}
The (subordinated) Poisson semigroup $\{\mathcal{P}_{t}\}_{t\geq 0}$ of a symmetric diffusion semigroup $\boldsymbol{\mathcal{T}}$ is defined by
\[
\mathcal{P}_{t}f=\frac{t}{2\sqrt{\pi}}\int_{0}^{\infty}s^{-\frac{3}{2}}e^{-\frac{t^2}{4s}}\mathcal{T}_{s}fds
\]
(see, e.g., \cite[p. 47]{Stein 1970-2}).
The typical examples of a symmetric diffusion semigroup and a subordinated Poisson semigroup
are the heat semigroup and the Poisson semigroup on $\mathbb{R}^{d}$, respectively.
As a generalization, the Littlewood-Paley $g$-function of the following form
can be considered for $k\in \mathbb{N}$ (natural numbers) and $1<q<\infty$:
\begin{equation}\label{eqn: LP g function}
g_{k,q,\mathcal{T}}(f)(x)=\left(\int_{0}^{\infty}\left|t^{k}\frac{\partial^{k}}{\partial t^{k}}\mathcal{T}_{t}f(x)\right|^q \frac{dt}{t}\right)^{\frac{1}{q}},
\end{equation}
and then in \cite{Stein 1970-2}, the following theorem has been proven.
\begin{theorem}
If $\{\mathcal{T}_{t}\}_{t\geq 0}$ is a symmetric diffusion semigroup,
then for any $1<p<\infty$ and $k\in\mathbb{N}$
there exists a constant $C>0$ depending on $p$ and $k$ such that for any $f\in L^{p}(\Omega,\mu)$,
\begin{equation}\label{eqn: Stein LP ineq for SDS}
C^{-1}\|f-E_{0}(f)\|_{L^{p}}\leq \|g_{k,2,\mathcal{T}}(f)\|_{L^{p}}\leq C\|f\|_{L^{p}},
\end{equation}
where $E_{0}(f)=\lim_{t\to\infty}\mathcal{T}_{t}f$.
\end{theorem}

The Littlewood-Paley $g$-function $g_{k,2,\mathcal{T}}$ is also important
for proving the maximal theorem and the multiplier theorem related to
the symmetric diffusion semigroup $\{\mathcal{T}_{t}\}_{t \geq 0}$ (see, e.g., \cite{Stein 1982}).
Since then, some authors studied the inequality \eqref{eqn: Stein LP ineq for SDS} for a Banach space-valued version of the function $g_{k,q,\mathcal{T}}$ defined in \eqref{eqn: LP g function}.
For a Banach space $(V,\|\cdot\|_{V})$ with a norm $\|\cdot\|_{V}$, $1<q<\infty$
and $k\in\mathbb{N}$, we define
\[
g_{k,q,\mathcal{S};V}(f)(x)
  =\left(\int_{0}^{\infty}
    \left\|t^{k}\frac{\partial^k}{\partial t^k}\mathcal{S}_{t}f(x)\right\|_{V}^q \frac{dt}{t}\right)^{\frac{1}{q}},
\]
where $\{\mathcal{S}_{t}\}_{t\ge0}$ is a symmetric diffusion semigroup, i.e., $\mathcal{S}_{t}=\mathcal{T}_{t}$ for all $t\ge0$,
or its subordinated Poisson semigroup.

In \cite{Martinez 2006},
the authors verified the relationship between the geometric property of a Banach space $V$
and the $L^{p}$-boundedness of the function $g_{1,q,\mathcal{P};V}$.
Recently, in \cite{Xu 2020}, the author extended the results obtained in \cite{Martinez 2006}
regarding the Littlewood-Paley $g$-function $g_{1,q,\mathcal{P};V}$
to the case of the Littlewood-Paley $g$-function $g_{k,q,\mathcal{T};V}$.
In \cite{Hytonen 2007}, the author proved the Littlewood-Paley-Stein inequality
for the Littlewood-Paley $g$-function defined by
the expectation of UMD Banach space-valued stochastic integrals with respect to Brownian motion,
which provides a characterization of UMD Banach spaces in terms of the Littlewood-Paley-Stein inequality.

On the other hand, in \cite{Torrea 2014},
the authors considered the fractional Littlewood-Paley $g$-function,
which is of a form that includes the fractional derivative of $\{\mathcal{P}_{t}\}_{t\geq 0}$.
Let $\alpha>0$ and $m$ be the smallest integer such that $m>\alpha$.
The fractional derivative of the subordinated semigroup $\{\mathcal{P}_{t}\}_{t\geq 0}$ is  defined by
\begin{align*}
\frac{\partial^{\alpha}}{\partial t^{\alpha}}\mathcal{P}_{t}f(x)
=\frac{e^{-i\pi (m-\alpha)}}{\Gamma(m-\alpha)}
  \int_{0}^{\infty}\frac{\partial^{m}}{\partial t^{m}}\mathcal{P}_{t+s}f(x)s^{m-\alpha-1}ds,
\end{align*}
and the fractional Littlewood-Paley $g$-function $g_{\alpha,q,\mathcal{P};V}$ is defined by
\[
g_{\alpha,q,\mathcal{P};V}(f)(x)=\left(\int_{0}^{\infty}\left\|t^{\alpha}\frac{\partial^\alpha}{\partial t^\alpha}\mathcal{P}_{t}f(x)\right\|_{V}^q \frac{dt}{t}\right)^{\frac{1}{q}}.
\]
Then the authors proved a characterization of Lusin type and cotype of Banach spaces
in terms of the Littlewood-Paley-Stein inequality (see \cite[Theorems A and B]{Torrea 2014}).

Let $\psi:[0,\infty)\times \mathbb{R}^{d}\rightarrow \mathbb{C}$ be a measurable function.
For each $t\geq 0$, the pseudo-differential operator $L_{\psi}(t)$ with the symbol $\psi$
is defined by
\begin{align*}
L_{\psi}(t)f(x)=\mathcal{F}^{-1}(\psi(t,\xi)\mathcal{F}f(\xi))(x)
\end{align*}
for suitable functions $f$ defined on $\mathbb{R}^{d}$,
where $\mathcal{F}$ and $\mathcal{F}^{-1}$ are the Fourier transform
and Fourier inverse transform, respectively (see \eqref{eqn:FTand FIT}).
For each $t>s$, we define the operator $\mathcal{T}_{\psi}(t,s)$ by
\begin{equation}\label{eqn: solution operator}
\mathcal{T}_{\psi}(t,s)f(x)=p_{\psi}(t,s,\cdot)*f(x)
\end{equation}
for suitable functions $f$, where
\begin{align*}
p_{\psi}(t,s,x):=\mathcal{F}^{-1}\left(\exp\left(\int_{s}^{t}\psi(r,\xi)dr\right)\right)(x).
\end{align*}
if the right hand side is well-defined, and so we have
\begin{align*}
\mathcal{T}_{\psi}(t,s)f(x)
=\mathcal{F}^{-1}\left(\exp\left(\int_{s}^{t}\psi(r,\xi)dr\right)\mathcal{F}f(\xi)\right)(x).
\end{align*}

In this paper, as a main theorem, we prove the Littlewood-Paley type inequality (see Theorem \ref{thm: Lp boundedness of S lambda infty})
for the pseudo-differential operator $L_{\psi_{1}}(l)$ and the evolution system $\{\mathcal{T}_{\psi_{2}}(t,s)\}_{t\ge s\ge0}$ for the pair $(\psi_{1},\psi_{2})\in\mathfrak{S}^{2}$ (see Section \ref{subsec:Symbols of PDO}).
More precisely, under certain conditions,
we prove that there exists a constant $C>0$ such that
for any $s,l\geq 0$ and $f\in L^{p}(\mathbb{R}^{d})$,
\begin{align*}
\int_{\mathbb{R}^{d}}\left(\int_{s}^{s+a}(t-s)^{\frac{q\gamma_{1}}{\gamma_{2}}-1}
   |L_{\psi_{1}}(l)\mathcal{T}_{\psi_{2}}(t,s)f(x)|^{q}dt\right)^{\frac{p}{q}}dx
\leq C\int_{\mathbb{R}^{d}}|f(x)|^{p}dx,
\end{align*}
where $\gamma_{1},\gamma_{2}$ are corresponding to the orders of the symbols $\psi_{1}$ and $\psi_{2}$, respectively
(see \textbf{(S1)} in Section \ref{subsec:Symbols of PDO}).
Furthermore, if $\psi_{1}(l,\xi)$ and $\psi_{2}(r,\xi)$ are constant with respect to the variables $l$ and $r$, respectively,
and $\psi_{1},\psi_{2}$ satisfy certain homogeneity (see \eqref{eqn:homogenioty of psi}),
then there exists a constant $C>0$ such that for any $s\geq 0$ and $f\in L^{p}(\mathbb{R}^{d})$,
\begin{align*}
\int_{\mathbb{R}^{d}}\left(\int_{s}^{\infty}(t-s)^{\frac{q\gamma_{1}}{\gamma_{2}}-1}
 |L_{\psi_{1}}\mathcal{T}_{\psi_{2}}(t,s)f(x)|^{q}dt\right)^{\frac{p}{q}}dx
\leq C\int_{\mathbb{R}^{d}}|f(x)|^{p}dx.
\end{align*}

This paper is organized as follows.
In Section \ref{sec: pre},
we discuss the integrability of the Fourier transforms of functions,
and recall the definitions and properties related to the pseudo-differential operator
and the operator $\mathcal{T}_{\psi}(t,s)$ with a symbol $\psi$.
In Section \ref{sec:Estimates},
we estimate the gradient of the kernel of convolution operators.
Especially, we prove that the kernel satisfies the H\"ormander's condition.
In Section \ref{sec:Convolution Operators},
we prove that the convolution operator corresponding to $L_{\psi_{1}}(l)\mathcal{T}_{\psi_{2}}(t,s)$
is a bounded linear operator from the Besov space on $\mathbb{R}^{d}$ into $L^{q}(\mathbb{R}^{d};V)$ for a Banach space $V$
(see Theorem \ref{lem: LP ineq for p equal lambda fixed s}).
In Section \ref{sec:Boundedness of Convolution Operators},
by applying the Calder\'{o}n-Zygmund theorem for vector-valued functions,
we prove the boundedness of the convolution operators
corresponding to $L_{\psi_{1}}(l)\mathcal{T}_{\psi_{2}}(t,s)$
(see Theorems \ref{thm: LP ineq p equal lambda} and \ref{thm: convolution is bdd on Lp}).
In Section \ref{sec:Main Results},
we prove the Littlewood-Paley type inequality for evolution systems associated with pseudo-differential operators
(see Theorem \ref{thm: Lp boundedness of S lambda infty}).

In \cite{Krylov 1994}, as a generalization of the Littlewood-Paley inequality,
the author proved a Littlewood-Paley type inequality for the heat semigroup acting on Hilbert space-valued $L^{p}$-spaces.
Since then, in \cite{I. Kim 2012}, the authors extended the generalized Littlewood-Paley inequality established in \cite{Krylov 1994}
to the fractional heat semigroup.
Furthermore, in \cite{I. Kim K.-H. Kim 2016},
the authors extended the results of the generalized Littlewood-Paley inequality obtained in \cite{I. Kim 2012}
to the case of evolution systems induced by the pseudo-differential operator with the symbol $\psi$ satisfying certain conditions.
The generalization of the Littlewood-Paley inequality in our setting and its extensions are currently in progress
and will appear in a separate paper.

\section{Symbols of Pseudo-Differential Operators}\label{sec: pre}

Throughout this paper, let $1\leq p< \infty$ and let $L^{p}(\mathbb{R}^{d})$ be the Banach space of all measurable functions
$f:\mathbb{R}^{d}\rightarrow\mathbb{R}$ such that
\[
\|f\|_{L^{p}(\mathbb{R}^{d})}:=\left(\int_{\mathbb{R}^{d}}|f(x)|^{p}dx\right)^{\frac{1}{p}}<\infty.
\]
For $f\in L^{1}(\mathbb{R}^{d})$, we denote the Fourier transform and inverse Fourier transform by
\begin{align}\label{eqn:FTand FIT}
\mathcal{F}(f)(\xi)&:=\frac{1}{(2\pi)^{d/2}}\int_{\mathbb{R}^{d}}e^{-i x\cdot\xi}f(x)dx,\nonumber\quad\\
\mathcal{F}^{-1}(f)(x)&:=\frac{1}{(2\pi)^{d/2}}\int_{\mathbb{R}^{d}}e^{i x\cdot\xi}f(\xi)d\xi.
\end{align}
Let $C_{c}^{\infty}(\mathbb{R}^{d})$ be the space of all infinitely differentiable functions with compact support.

\subsection{Integrability of Fourier Transform}

We note that for any $w\in \mathbb{R}^{d}$ (for $d\in\mathbb{N}$) with $|w|=1$, and $\eta>0$, $\zeta<\eta$,
\begin{align}\label{eqn: integral 1}
\int_{|z|\geq \frac{1}{2}}\frac{|w+z|^{\zeta}}{|z|^{d+\eta}}dz &<\infty.
\end{align}

\begin{lemma}\label{lem: f-lemma 2}
Let $x\in\mathbb{R}^{d}$ (for $d\in\mathbb{N}$) and $\theta\in [0,1]$ be given. Let $f:[0,\infty)\rightarrow\mathbb{R}$ be a decreasing function. Then for any $\zeta\in\mathbb{R}$ and $\delta<0$, there exists a constant $C>0$ depending on $d,\zeta$ and $\eta$ such that
\[
\int_{|y|<\frac{|x|}{2} }|x+\theta y|^{\zeta}f(|x+\theta y|)\frac{1}{|y|^{d+\delta}}dy
\leq C|x|^{\zeta-\delta}f\left(\frac{|x|}{2}\right).
\]
\end{lemma}

\begin{proof}
Note that if $|y|\leq \frac{|x|}{2}$ and $\theta\in [0,1]$, then we see that
\begin{align}
\frac{1}{2}|x|\le |x+\theta y|\le \frac{3}{2}|x|.\label{eqn: x+ theta y lower bound}
\end{align}
Therefore, for any $\zeta\in\mathbb{R}$, since $f$ is decreasing, we have
\begin{align}\label{eqn:bdd of theta-f1}
|x+\theta y|^{\zeta}f(|x+\theta y|)\leq C|x|^{\zeta}f\left(\frac{|x|}{2}\right).
\end{align}
On the other hand, by changing variables by $y=|x|z$ and the change of variables to polar coordinates, we see that
\begin{align}\label{eqn:int by polar}
\int_{|y|<\frac{|x|}{2}}\frac{1}{|y|^{d+\delta}}dy
\leq \sigma(\partial B_{1})\frac{1}{-\delta}\left(\frac{1}{2}\right)^{-\delta}|x|^{\zeta-\delta},
\end{align}
where $\partial B_{1}$ and $d\sigma$ are the unit sphere and the surface measure on $\partial B_{1}$, respectively.
Therefore, by combining \eqref{eqn:bdd of theta-f1} and \eqref{eqn:int by polar}, we have the desired assertion.
\end{proof}

Let $\eta\geq 0$ and $\zeta>\eta-\frac{d}{2}$ (for $d\in\mathbb{N}$) be given.
We now consider a class $\boldsymbol{F}_{\eta,\zeta}$ of all nonnegative decreasing integrable functions
$f: [0,\infty)\rightarrow \mathbb{R}$ satisfying that
there exist positive constants $T,C_{f},\alpha_{f},\beta_{f}>0$ with $T\ge1$, $\alpha:=\alpha_{f}>\zeta+d$ and
$\max\{\zeta-\eta,0\}<\beta:=\beta_{f}<\zeta-\eta+\frac{d}{2}$ such that
\begin{align}
f(t)&\leq C_{f}(1+t)^{-\beta},\quad\text{for}\quad 0\leq t <T,\label{eqn: f condition 1}\\
f(t)&\leq C_{f}t^{-\alpha},\quad\text{for}\quad t\geq T.\label{eqn: f condition 1-1}
\end{align}
In fact, from the above conditions, if $f\in\boldsymbol{F}_{\eta,\zeta}$, then it is clear that
\begin{align*}
0<\beta<\zeta-\eta+\frac{d}{2}<\zeta+d<\alpha,
\end{align*}
which implies that for all $t\ge T$, $f(t)\leq C_{f}t^{-\alpha}\le C_{f}t^{-\beta}$, and for all $0<t <T$,
\begin{align*}
f(t)&\leq C_{f}(1+t)^{-\beta}\le C_{f} t^{-\beta}.
\end{align*}
Therefore, for all $t>0$, it holds that
\begin{align}\label{eqn:PB of f}
f(t)\le C_{f} t^{-\beta}.
\end{align}

\begin{example}
\upshape
Let $f$ be a nonnegative decreasing integrable function defined on $[0,\infty)$
satisfying that for any $\sigma>0$, there exist constants $K,c>0$ such that for all $t\geq0$,
\begin{align}\label{eqn: f-condition}
  t^{\sigma}f(t)\leq Kf(ct).
\end{align}
Then we can easily see that $f\in \boldsymbol{F}_{\eta,\zeta}$.
As a concrete example of elements in $\boldsymbol{F}_{\eta,\zeta}$, we can consider a function $f(t)=e^{-t^{\gamma}}$ ($t\geq 0$) for some $\gamma>0$.
Then we can easily see that $f$ satisfies the condition given in \eqref{eqn: f-condition} and thus $f\in\boldsymbol{F}_{\eta,\zeta}$.
\end{example}

\begin{lemma}\label{lem: fractional laplacian is in L2-0}
Let $\eta\ge0$ and let $h:\mathbb{R}^{d}\to\mathbb{R}$ be a measurable function.
Suppose that there exist constants $C_{1}>0$, $\zeta>\eta-\frac{d}{2}$ and $f\in \boldsymbol{F}_{\eta,\zeta}$ such that
\begin{align}\label{eqn:bddness of h}
|h(x)|&\leq C_{1} |x|^{\zeta}f(|x|)\quad \text{for all } x\in \mathbb{R}^{d}\setminus\widetilde{\boldsymbol{0}},
\end{align}
where $\widetilde{\boldsymbol{0}}=\{(x_1,\cdots,x_d)|x_i=0\text{ for some }i=1,\cdots,d\}$.
Then we have
\begin{align}\label{eqn:L2-int of h with eta}
\int_{\mathbb{R}^{d}}|x|^{-2\eta}|h(x)|^{2}dx<N<\infty,
\end{align}
where $N=N(C_{1},C_{f},\eta,\zeta,d)$.
\end{lemma}

\begin{proof}
By applying \eqref{eqn:bddness of h}, the change of variables to polar coordinates and applying the assumptions that $f$ is decreasing and \eqref{eqn: f condition 1-1},
we obtain that
\begin{align}
\int_{\mathbb{R}^{d}}|x|^{-2\eta}|h(x)|^2dx
&\leq C_{1}^2\int_{\mathbb{R}^{d}}|x|^{2\zeta-2\eta}f(|x|)^{2}dx
=C_{1}^2\sigma(\partial B_{1})\int_{0}^{\infty}\rho^{2\zeta-2\eta+d-1}f(\rho)^{2}d\rho\nonumber\\
&\leq C_{1}^2\sigma(\partial B_{1})f(0)^{2}\int_{0}^{T}\rho^{2\zeta-2\eta+d-1}d\rho\nonumber\\
&\qquad +\sigma(\partial B_{1})C_{1}^{2}C_{f}^2\int_{T}^{\infty}\rho^{2\zeta-2\eta+d-1}\rho^{-2\alpha}d\rho\nonumber\\
&<\infty.\label{eqn:L2-int of h}
\end{align}
\end{proof}

\begin{lemma}\label{lem: fractional laplacian is in L2}
Let $\eta\in [0,2)$ and let $h\in C^{2}(\mathbb{R}^{d}\setminus\widetilde{\boldsymbol{0}})$.
Suppose that there exist constants $C_{1}>0$, $\zeta>\eta-\frac{d}{2}$ and $f\in \boldsymbol{F}_{\eta,\zeta}$ such that
\begin{align}\label{eqn: condition 2}
|h(x)|&\leq C_{1} |x|^{\zeta}f(|x|)\quad \text{for all } x\in \mathbb{R}^{d}\setminus\widetilde{\boldsymbol{0}}.
\end{align}
Further we assume that
\begin{itemize}
  \item [\rm{(i)}] if $\eta\in (0,1)$, then there exists a constant $C_{2}>0$
  such that for all $x\in \mathbb{R}^{d}\setminus\widetilde{\boldsymbol{0}}$,
\begin{equation}\label{eqn: condition 3}
|\nabla h(x)|\leq C_{2}|x|^{\zeta-1}f(|x|),
\end{equation}
where the absolute value of the left hand sides of \eqref{eqn: condition 3} means the Euclidean norm,
  \item [\rm{(ii)}] if $\eta\in [1,2)$, then there exists a constant $C_{3}>0$
  such that for all $x\in \mathbb{R}^{d}\setminus\widetilde{\boldsymbol{0}}$,
\begin{equation}\label{eqn: condition 4}
|\nabla(\nabla h)(x)|\leq C_{3}|x|^{\zeta-2}f(|x|),
\end{equation}
where $\nabla(\nabla f):=\left[\frac{\partial^{2}f}{\partial x_{i} \partial x_{j}}\right]_{i,j=1}^{d}$.
\end{itemize}
Then we have
\begin{align}\label{eqn:L2-int of Delta-eta-h}
\int_{\mathbb{R}^{d}}|(-\Delta)^{\frac{\eta}{2}}h(x)|^{2}dx<N<\infty,
\end{align}
where $N=N(C_{1},C_{2},C_{3},C_{f},\eta,\zeta,\alpha_f,\beta_f,d)$.
\end{lemma}

\begin{proof}
The proof is a simple modification of the proof of Lemma 5.1 in \cite{I. Kim K.-H. Kim 2015}.
However, for the reader's convenience, we will give the full proof.
The case of $\eta=0$ is clear from Lemma \ref{lem: fractional laplacian is in L2-0}.
We now consider the case $\eta\in (0,2)$. Let $\partial B_{1}$ and $d\sigma$ be the unit sphere and the surface measure on $\partial B_{1}$, respectively.

(i) \enspace
We consider the fractional power $(-\Delta)^{\frac{\eta}{2}}$ of the Laplacian
as defined by the principal value integral (with $C=C(\eta)>0$):
\begin{align*}
-(-\Delta)^{\frac{\eta}{2}}h(x)
&=C\lim_{\epsilon\rightarrow 0}\int_{|y|\geq \epsilon}\frac{h(x+y)-h(x)}{|y|^{d+\eta}}dy
=C\int_{\mathbb{R}^{d}}\frac{h(x+y)-h(x)}{|y|^{d+\eta}}dy.
\end{align*}
Then for any $x\in \mathbb{R}^{d}$, we obtain that
\begin{align*}
\left|(-\Delta)^{\frac{\eta}{2}}h(x)\right|
&\le C\int_{|y|\ge \frac{|x|}{2}}\frac{|h(x+y)-h(x)|}{|y|^{d+\eta}}dy
  +C\int_{|y|<\frac{|x|}{2}}\frac{|h(x+y)-h(x)|}{|y|^{d+\eta}}dy\\
&\le \mathcal{I}(x)+C|h(x)|\int_{|y|\ge \frac{|x|}{2}}\frac{1}{|y|^{d+\eta}}dy+\mathcal{J}(x)\\
&\le C|\mathcal{I}(x)|+C|x|^{-\eta}|h(x)|\int_{|y|\ge \frac{1}{2}}\frac{1}{|y|^{d+\eta}}dy+C|\mathcal{J}(x)|,
\end{align*}
where
\begin{align*}
\mathcal{I}(x)=\int_{|y|\geq \frac{|x|}{2}}\frac{h(x+y)}{|y|^{d+\eta}}dy,\quad
\mathcal{J}(x)=\int_{|y|<\frac{|x|}{2}}\frac{h(x+y)-h(x)}{|y|^{d+\eta}}dy.
\end{align*}
Therefore, by applying \eqref{eqn: integral 1} and \eqref{eqn:L2-int of h with eta},
it is enough to prove that
\begin{align*}
I&=\int_{|x|<1}|\mathcal{I}(x)|^{2}dx+\int_{|x|\geq 1}|\mathcal{I}(x)|^{2}dx=:I_{1}+I_{2}<\infty,\\
J&=\int_{\mathbb{R}^{d}}|\mathcal{J}(x)|^{2}dx<\infty.
\end{align*}
For $I_{1}<\infty$, by applying \eqref{eqn: condition 2} and \eqref{eqn:PB of f},
and changing variables by $y=|x|z$ and applying \eqref{eqn: integral 1} with $\zeta-\beta<\eta$, we obtain that
\begin{align*}
|\mathcal{I}(x)|
&\leq C_{1}\int_{|y|\geq \frac{|x|}{2}}\frac{|x+y|^{\zeta}f(|x+y|)}{|y|^{d+\eta}}dy\nonumber\\
&\leq C_{1}C_{f}\int_{|y|\geq \frac{|x|}{2}}\frac{|x+y|^{\zeta-\beta}}{|y|^{d+\eta}}dy\nonumber\\
&=C_{1}C_{f}|x|^{-\eta}\int_{|z|\geq \frac{1}{2}}\frac{|x+|x|z|^{\zeta-\beta}}{|z|^{d+\eta}}dz\nonumber\\
&= C_{1}C_{f}|x|^{\zeta-\beta-\eta}\int_{|z|\geq \frac{1}{2}}\frac{|x/|x|+z|^{\zeta-\beta}}{|z|^{d+\eta}}dz\nonumber\\
&\leq C'|x|^{\zeta-\beta-\eta},\quad
   C'=C_{1}C_{f}\sup_{|w|=1}\int_{|z|\geq \frac{1}{2}}\frac{|w+z|^{\zeta-\beta}}{|z|^{d+\eta}}dz<\infty.
\end{align*}
Therefore, by the condition that $\beta<\zeta-\eta+\frac{d}{2}$, we have $2\zeta-2\beta-2\eta+d>0$ and then
\begin{align}\label{eqn: estimate 1}
I_{1}&=\int_{|x|< 1}|\mathcal{I}(x)|^{2}dx
\leq C'^{2}\int_{|x|< 1}|x|^{2\zeta-2\beta-2\eta}dx<\infty.
\end{align}
For $I_{2}<\infty$, by applying \eqref{eqn: condition 2}, we obtain that
\begin{align*}
|\mathcal{I}(x)|
&\leq C_{1}\int_{|y|\geq \frac{|x|}{2}}\frac{|x+y|^{\zeta}f(|x+y|)}{|y|^{d+\eta}}dy\\
&\le 2^{d+\eta}C_{1}|x|^{-(d+\eta)}\int_{|y|\geq \frac{|x|}{2}} |x+y|^{\zeta}f(|x+y|)dy\nonumber\\
&\leq D_{1}|x|^{-(d+\eta)},
\end{align*}
where since $\alpha>\zeta+d$, by applying the fact that $f$ is decreasing and \eqref{eqn: f condition 1-1}, we obtain that
\begin{align*}
D_{1}&=2^{d+\eta}C_{1}\int_{\mathbb{R}^{d}}|y|^{\zeta}f(|y|)dy\\
&\leq 2^{d+\eta}C_{1}\sigma(\partial B_{1})
     \left(f(0)\int_{0}^{T}\rho^{\zeta+d-1}d\rho+C_{f}\int_{T}^{\infty}\rho^{\zeta+d-1}\rho^{-\alpha}d\rho\right)\nonumber\\
&<\infty.
\end{align*}
Therefore, we have
\begin{align}\label{eqn: estimate 2}
I_{2}&=\int_{|x|\geq  1}|\mathcal{I}(x)|^{2}dx
 \leq D_{1}^{2}\int_{|x|\geq 1}|x|^{-2(d+\eta)}dx
<\infty.
\end{align}
For $J<\infty$, by using Taylor's theorem, \eqref{eqn: condition 3} and Lemma \ref{lem: f-lemma 2} with $\delta=\eta-1$, for some $\theta\in [0,1]$, we obtain that
\begin{align*}
|\mathcal{J}(x)|
&\leq \int_{|y|<\frac{|x|}{2}}|\nabla h(x+\theta y)|\frac{1}{|y|^{d-1+\eta}}dy\nonumber\\
&\leq C_{2}\int_{|y|<\frac{|x|}{2}}|x+\theta y|^{\zeta-1}f(|x+\theta y|)\frac{1}{|y|^{d-1+\eta}}dy\nonumber\\
&\leq C'|x|^{(\zeta-1)-(\eta-1)}f\left(\frac{|x|}{2}\right)\\
&=C'|x|^{\zeta-\eta}f\left(\frac{|x|}{2}\right).
\end{align*}
By the similar estimations used in \eqref{eqn:L2-int of h}, we have
\begin{align}\label{eqn: estimate 4}
J=\int_{\mathbb{R}^{d}}|\mathcal{J}(x)|^{2}dx
&\leq (C')^2\int_{\mathbb{R}^{d}}|x|^{2\zeta-2\eta}f\left(\frac{|x|}{2}\right)^{2}dx<\infty.
\end{align}
Therefore, by \eqref{eqn: estimate 1}, \eqref{eqn: estimate 2} and \eqref{eqn: estimate 4}, we obtain
\begin{align*}
\int_{\mathbb{R}^{d}}|(-\Delta)^{\frac{\eta}{2}}h(x)|^{2}dx<\infty.
\end{align*}

(ii) \enspace
Suppose \eqref{eqn: condition 4} holds.
Then in the proof of (i), we proved that
\begin{align*}
\int_{\mathbb{R}^{d}}|(-\Delta)^{\frac{\eta}{2}}h(x)|^{2}dx
\le D+3C^{2}\left(I_{1}+I_{2}+J\right)
\end{align*}
for some constant $D$, where $I_{1}$ and $I_{2}$ are independent of the choice of $\eta$.
Therefore, it is enough to prove that $J<\infty$ for the case of $\eta\in [1,2)$.
Note that by the change of variables, we have
\begin{align*}
\mathcal{J}(x)=\int_{\frac{|x|}{2}>|y|}\frac{h(x+y)-h(x)}{|y|^{d+\eta}}dy
=\int_{\frac{|x|}{2}>|y|}\frac{h(x-y)-h(x)}{|y|^{d+\eta}}dy,
\end{align*}
and so we have
\[
2\mathcal{J}(x)=\mathcal{J}(x)+\mathcal{J}(x)
=\int_{\frac{|x|}{2}>|y|}\frac{h(x+y)+h(x-y)-2h(x)}{|y|^{d+\eta}}dy.
\]
 In order to estimate $\mathcal{J}$, we use the second order Taylor's theorem.
 Then by applying \eqref{eqn: condition 4} and Lemma \ref{lem: f-lemma 2} with $\delta=\eta-2$, for some $\theta\in [0,1]$, we obtain that
\begin{align}\label{eqn: estimate 4 second-1}
|2\mathcal{J}(x)|
&\leq \int_{|y|<\frac{|x|}{2}}|\nabla(\nabla h)(x+\theta y)|\frac{1}{|y|^{d-2+\eta}}dy\nonumber\\
&\leq C_{3}\int_{|y|<\frac{|x|}{2}}|x+\theta y|^{\zeta-2}f(|x+\theta y|)\frac{1}{|y|^{d-2+\eta}}dy\nonumber\\
&\leq C'|x|^{(\zeta-2)-(\eta-2)}f\left(\frac{|x|}{2}\right)\nonumber\\
&=C'|x|^{\zeta-\eta}f\left(\frac{|x|}{2}\right).
\end{align}
Then by the same arguments used in \eqref{eqn:L2-int of h}, we have
\begin{equation}\label{eqn: esti J eta in (1,2)}
\int_{\mathbb{R}^{d}}|\mathcal{J}(x)|^2 dx
\leq \frac{1}{4}(C')^2\int_{\mathbb{R}^{d}}|x|^{2\zeta-2\eta}f\left(\frac{|x|}{2}\right)^2 dx<\infty.
\end{equation}
Using \eqref{eqn: estimate 1}, \eqref{eqn: estimate 2} and  \eqref{eqn: esti J eta in (1,2)},  we have
\[
\int_{\mathbb{R}^{d}}|(-\Delta)^{\frac{\eta}{2}}h(x)|^{2}dx<\infty.
\]
The proof is now complete.
\end{proof}

The following Lemma is a generalization of Lemma 5.1 in \cite{I. Kim S. Lim K.-H. Kim 2016}.

\begin{lemma}\label{lem: fractional laplacian is in L2 eta geq 2}
Let $\eta\geq 0$ and let
$h\in C^{1+\lfloor\eta\rfloor}(\mathbb{R}^{d}\setminus\widetilde{\boldsymbol{0}})$,
where $\lfloor\eta\rfloor$ is the greatest integer less than or equal to $\eta$.
Suppose that there exist constants $C_{1}>0,\zeta>\eta-\frac{d}{2}$  and $f\in\boldsymbol{F}_{\eta,\zeta}$ such that
\begin{align}
|(-\Delta)^{\lfloor\frac{\eta}{2}\rfloor}h(x)|&\leq C_{1} |x|^{\zeta-2\lfloor\frac{\eta}{2}\rfloor}f(|x|)\quad \text{for all}\quad x\in \mathbb{R}^{d}\setminus\widetilde{\boldsymbol{0}}\label{eqn: condition 2-2}.
\end{align}
Further we assume that
\begin{itemize}
  \item [\rm{(i)}] if $\eta-2\lfloor\frac{\eta}{2}\rfloor\in (0,1)$, then there exists a constant $C_{2}>0$ such that for all $x\in \mathbb{R}^{d}\setminus\widetilde{\boldsymbol{0}}$,
  \begin{equation}\label{eqn: condition 3-2}
|\nabla (-\Delta)^{\lfloor\frac{\eta}{2}\rfloor} h(x)|\leq C_{2}|x|^{\zeta-2\lfloor\frac{\eta}{2}\rfloor-1}f(|x|),
\end{equation}
  \item [\rm{(ii)}] if $\eta-2\lfloor\frac{\eta}{2}\rfloor\in [1,2)$, then there exists a constant $C_{3}>0$ such that for all $x\in \mathbb{R}^{d}\setminus\widetilde{\boldsymbol{0}}$,
  \begin{equation}\label{eqn: condition 4-2}
|\nabla(\nabla (-\Delta)^{\lfloor\frac{\eta}{2}\rfloor}h)(x)|\leq C_{3}|x|^{\zeta-2\lfloor\frac{\eta}{2}\rfloor-2}f(|x|),
\end{equation}
where $\nabla(\nabla f):=\left[\frac{\partial^{2}f}{\partial x_{i} \partial x_{j}}\right]_{i,j=1}^{d}$.
\end{itemize}
Then we have
\[
\int_{\mathbb{R}^{d}}|(-\Delta)^{\frac{\eta}{2}}h(x)|^{2}dx<N<\infty,
\]
where $N=N(C_{1},C_{2},C_{3},\eta,\beta,d)$.
\end{lemma}

\begin{proof}
The proof is exactly the same as the proof of Lemma 5.1 in \cite{I. Kim S. Lim K.-H. Kim 2016}.
\end{proof}

\begin{theorem}\label{thm:ph is in L1}
Let $\eta> \frac{d}{2}$ and let
$h\in C^{1+\lfloor\eta\rfloor}(\mathbb{R}^{d}\setminus\widetilde{\boldsymbol{0}})$.
Suppose that there exist constants $C_{1}>0,\zeta>\eta-\frac{d}{2}$  and $f\in\boldsymbol{F}_{\eta,\zeta}$ such that
\begin{align}
|(-\Delta)^{\lfloor\frac{\eta}{2}\rfloor}h(x)|&
\leq C_{1} |x|^{\zeta-2\lfloor\frac{\eta}{2}\rfloor}f(|x|)\quad \text{for all}\quad x\in \mathbb{R}^{d}\setminus\widetilde{\boldsymbol{0}}\label{eqn: condition 2-3}.
\end{align}
Further we assume that
\begin{itemize}
  \item [\rm{(i)}] if $\eta-2\lfloor\frac{\eta}{2}\rfloor\in (0,1)$, then there exists a constant $C_{2}>0$ such that for all $x\in \mathbb{R}^{d}\setminus\widetilde{\boldsymbol{0}}$,
  \begin{equation}\label{eqn: condition 3-3}
|\nabla (-\Delta)^{\lfloor\frac{\eta}{2}\rfloor} h(x)|\leq C_{2}|x|^{\zeta-2\lfloor\frac{\eta}{2}\rfloor-1}f(|x|),
\end{equation}
  \item [\rm{(ii)}] if $\eta-2\lfloor\frac{\eta}{2}\rfloor\in [1,2)$, then there exists a constant $C_{3}>0$ such that for all $x\in \mathbb{R}^{d}\setminus\widetilde{\boldsymbol{0}}$,
  \begin{equation}\label{eqn: condition 4-3}
|\nabla(\nabla (-\Delta)^{\lfloor\frac{\eta}{2}\rfloor}h)(x)|\leq C_{3}|x|^{\zeta-2\lfloor\frac{\eta}{2}\rfloor-2}f(|x|),
\end{equation}
where $\nabla(\nabla f):=\left[\frac{\partial^{2}f}{\partial x_{i} \partial x_{j}}\right]_{i,j=1}^{d}$.
\end{itemize}
If $h\in L^{2}(\mathbb{R}^{d})$, then the following function
\[
p_{h}(x):=\mathcal{F}^{-1}(h(\xi))(x).
\]
is integrable on $\mathbb{R}^{d}$.
\end{theorem}

\begin{proof}
The proof is a simple modification of the argument used in the proof of Lemma 5.2 in \cite{I. Kim S. Lim K.-H. Kim 2016}.
In fact, since $\eta>\frac{d}{2}$, we have
\begin{align*}
A_{d}:=\int_{\mathbb{R}^{d}}\frac{1}{(1+|y|^{\eta})^{2}}dy<\infty.
\end{align*}
Then by the Cauchy-Schwarz inequality, Plancherel theorem, triangle inequality,
and Lemma \ref{lem: fractional laplacian is in L2 eta geq 2}, we obtain that
\begin{align}
\|p_{h}\|_{L^{1}(\mathbb{R}^{d})}
&=\int_{\mathbb{R}^{d}}|\mathcal{F}^{-1}\left(h(\xi)\right)(x)|dx\nonumber\\
&=\int_{\mathbb{R}^{d}}|(1+|x|^{\eta})^{-1}(1+|x|^{\eta})\mathcal{F}^{-1}\left(h(\xi)\right)(x)|dx\nonumber\\
&\leq \left(\int_{\mathbb{R}^{d}}(1+|x|^{\eta})^{-2}dx\right)^{\frac{1}{2}}
  \left(\int_{\mathbb{R}^{d}}\left|(1+|x|^{\eta})\mathcal{F}^{-1}\left(h(\xi)\right)(x)\right|^{2}dx\right)^{\frac{1}{2}}\nonumber\\
&= A_{d}^{\frac{1}{2}}\left(\int_{\mathbb{R}^{d}}
    \left|h(\xi)+(-\Delta_{\xi})^{\frac{\eta}{2}}h(\xi)\right|^{2}d\xi\right)^{\frac{1}{2}}\nonumber\\
&\leq A_{d}^{\frac{1}{2}}\left(\|h\|_{L^{2}(\mathbb{R}^{d})}+\|(-\Delta_{\xi})^{\frac{\eta}{2}}h\|_{L^{2}(\mathbb{R}^{d})}\right)\nonumber\\
&<\infty.\label{eqn: L1norm of ph}
\end{align}
The proof is complete.
\end{proof}

The following corollary will be used for the proof of Lemma \ref{lem: boundedness of L psi of p j}.

\begin{corollary}\label{cor: integrability bounded interval}
Let $\eta\geq 0$ and let
$h\in C^{1+\lfloor\eta\rfloor}(\mathbb{R}^{d}\setminus\widetilde{\boldsymbol{0}})$. Let $0\leq a< b<\infty$.
Suppose that there exist constants $C_{1}>0,\zeta>\eta-\frac{d}{2}$  and $f\in\boldsymbol{F}_{\eta,\zeta}$ such that
\begin{align}
|(-\Delta)^{\lfloor\frac{\eta}{2}\rfloor}h(x)|&\leq C_{1} |x|^{\zeta-2\lfloor\frac{\eta}{2}\rfloor}f(|x|)1_{(a,b)}(|x|)\quad \text{for all}\quad x\in \mathbb{R}^{d}\setminus\widetilde{\boldsymbol{0}}\label{eqn: condition 2-2-1}.
\end{align}
Further, we assume that
\begin{itemize}
  \item [\rm{(i)}] if $\eta-2\lfloor\frac{\eta}{2}\rfloor\in (0,1)$, then there exists a constant $C_{2}>0$ such that for all $x\in \mathbb{R}^{d}\setminus\widetilde{\boldsymbol{0}}$,
  \begin{equation}\label{eqn: condition 3-2-1}
|\nabla (-\Delta)^{\lfloor\frac{\eta}{2}\rfloor} h(x)|\leq C_{2}|x|^{\zeta-2\lfloor\frac{\eta}{2}\rfloor-1}f(|x|)1_{(a,b)}(|x|),
\end{equation}
  \item [\rm{(ii)}] if $\eta-2\lfloor\frac{\eta}{2}\rfloor\in [1,2)$, then there exists a constant $C_{3}>0$ such that for all $x\in \mathbb{R}^{d}\setminus\widetilde{\boldsymbol{0}}$,
\begin{equation}\label{eqn: condition 4-2-1}
|\nabla(\nabla (-\Delta)^{\lfloor\frac{\eta}{2}\rfloor}h)(x)|\leq C_{3}|x|^{\zeta-2\lfloor\frac{\eta}{2}\rfloor-2}f(|x|)1_{(a,b)}(|x|),
\end{equation}
where $\nabla(\nabla f):=\left[\frac{\partial^{2}f}{\partial x_{i} \partial x_{j}}\right]_{i,j=1}^{d}$.
\end{itemize}
Then it holds that
\begin{itemize}
  \item [\rm{(a)}] there exists a constant $C_{4}>0$ such that
  \begin{equation}\label{eqn: L2 norm of frac. Laplacian f(a)}
\|(-\Delta)^{\frac{\eta}{2}}h\|_{L^{2}(\mathbb{R}^{d})}\leq C_{4}f(a),
\end{equation}
  \item [\rm{(b)}] if $\eta>\frac{d}{2}$ and $\|h\|_{L^{2}(\mathbb{R}^{d})}\leq C_{5}f(a)$ for some constant $C_{5}>0$, then the function
  \[
  p_{h}(x)=\mathcal{F}^{-1}(h(\xi))(x)
  \]
  is integrable on $\mathbb{R}^{d}$ and
  \begin{equation}\label{eqn:ph L1 f(a)}
  \|p_{h}\|_{L^{1}(\mathbb{R}^{d})}\leq C_{6} f(a)
 \end{equation}
 for some constant $C_{6}>0$.
\end{itemize}
\end{corollary}

\begin{proof}
(a) \enspace
From the conditions \eqref{eqn: condition 2-2-1}, \eqref{eqn: condition 3-2-1}, \eqref{eqn: condition 4-2-1}
and the fact that $f$ is decreasing, we obtain that
\begin{align*}
|(-\Delta)^{\lfloor\frac{\eta}{2}\rfloor}h(x)|
&\leq C_{1}f(a)|x|^{\zeta-2\lfloor\frac{\eta}{2}\rfloor}1_{[0,b)}(|x|),\\
|\nabla (-\Delta)^{\lfloor\frac{\eta}{2}\rfloor}h(x)|
&\leq C_{2}f(a)|x|^{\zeta-2\lfloor\frac{\eta}{2}\rfloor-1}1_{[0,b)}(|x|),\\
|\nabla(\nabla (-\Delta)^{\lfloor\frac{\eta}{2}\rfloor}h)(x)|
&\leq C_{3}f(a)|x|^{\zeta-2\lfloor\frac{\eta}{2}\rfloor-2}1_{[0,b)}(|x|).
\end{align*}
Let $g(t)=1_{[0,b)}(t)$ ($t\geq 0$). Then $g$ is nonnegative decreasing on $[0,\infty)$.
By taking $T$ with $T\ge 1$, any real number $\beta,\alpha$
such that $\max\{\zeta-\eta,0\}<\beta<\zeta-\eta+\frac{d}{2}$ and $\alpha>\zeta+d$,
we see that
\begin{align*}
g(t)&\leq C(1+t)^{-\beta},\quad \text{for}\quad 0\leq t< T,\\
g(t)&\leq Ct^{-\alpha},\quad \text{for}\quad t\geq T
\end{align*}
for some $C>0$, which implies that $g\in\boldsymbol{F}_{\eta,\zeta}$.
Then by applying the arguments used in the proof of Lemma \ref{lem: fractional laplacian is in L2}
and Lemma \ref{lem: fractional laplacian is in L2 eta geq 2},
we can prove that \eqref{eqn: L2 norm of frac. Laplacian f(a)} holds.

(b) \enspace
By \eqref{eqn: L1norm of ph}, \eqref{eqn: L2 norm of frac. Laplacian f(a)}
and the assumption $\|h\|_{L^{2}(\mathbb{R}^{d})}\leq C_{5}f(a)$, we have
\[
\|p_{h}\|_{L^{1}(\mathbb{R}^{d})}
\leq A_{d}^{\frac{1}{2}}\left(\|h\|_{L^{2}(\mathbb{R}^{d})}+\|(-\Delta_{\xi})^{\frac{\eta}{2}}h\|_{L^{2}(\mathbb{R}^{d})}\right)
\leq A_{d}^{\frac{1}{2}}\left(C_{5}+C_{4}\right)f(a).
\]
The proof is complete.
\end{proof}

\subsection{Symbols of Pseudo-Differential Operators}
\label{subsec:Symbols of PDO}

For $t\geq 0$ and $f\in C_{c}^{\infty}(\mathbb{R}^{d})$,
the pseudo-differential operator $L_{\psi}(t)$ with symbol $\psi$ is defined by
\begin{align}\label{eqn: pseudo-differential operator}
L_{\psi}(t)f(x)=\mathcal{F}^{-1}(\psi(t,\xi)\mathcal{F}f(\xi))(x)
\end{align}
if the inverse Fourier transform exists. Then for any $k\in \mathbb{N}$ and $t\ge0$,
if $f$ belongs to the domain of $L_{\psi}(t)^{k}$, then it is obvious that
\begin{align*}
L_{\psi}(t)^{k}f(x)=L_{\psi^{k}}(t)f(x).
\end{align*}

We now consider the following conditions for symbols $\psi$ of pseudo-differential operators.
There exist positive constants $\kappa:=\kappa_{\psi},\mu:=\mu_{\psi},\gamma:=\gamma_{\psi}>0$
and a natural number $N:=N_{\psi}\in\mathbb{N}$ with $N\ge \lfloor\frac{d}{2}\rfloor +1$ such that
\begin{itemize}
  \item [\textbf{(S1)}] for almost all $t\ge0$ and $\xi\in\mathbb{R}^{d}$ (with respect to the Lebesgue measure),
\begin{equation*}
{\rm Re}[\psi(t,\xi)]\leq -\kappa|\xi|^{\gamma},\quad
\end{equation*}
where ${\rm Re}(z)$ is the real part of the complex number $z$,
  \item [\textbf{(S2)}] for any multi-index $\alpha=(\alpha_1,\cdots,\alpha_d)\in \mathbb{N}_{0}^{d}$ with
  $|\alpha|:=\alpha_1+\cdots+\alpha_d\leq N$, where $\mathbb{N}_{0}=\mathbb{N}\cup\{0\}$,
   and for almost all $t\ge0$ and $\xi=(\xi_1,\cdots,\xi_d)\in\mathbb{R}^{d}\setminus\widetilde{\boldsymbol{0}}$,
      \begin{equation*}
      |\partial_{\xi}^{\alpha}\psi(t,\xi)|\leq \mu|\xi|^{\gamma-|\alpha|},
      \end{equation*}
where $\partial_{\xi}^{\alpha}=\partial_{\xi_1}^{\alpha_1}\cdots \partial_{\xi_d}^{\alpha_d}$
and $\partial_{\xi_i}^{\alpha_i}$ is the $\alpha_i$-th derivative in the variable $\xi_i$.
\end{itemize}
Let $\mathfrak{S}$ be the set of all measurable functions (symbols of pseudo-differential operators)
$\psi:[0,\infty)\times \mathbb{R}^{d}\rightarrow \mathbb{C}$ satisfying the conditions \textbf{(S1)} and \textbf{(S2)}.

\begin{remark}
\upshape
Similar conditions to \textbf{(S1)} and \textbf{(S2)} have been considered in \cite{I. Kim 2016} and \cite{Kumano-go 1973}.
In particular, for the condition used in \cite{I. Kim 2016}, we can choose $N=\lfloor \frac{d}{2}\rfloor+1$.
However, for our study we need a more general choice of $N$ which depends on the degree of the symbol stated in \textbf{(S1)}
and will be used for our main theorem
(see Theorem \ref{thm: Lp boundedness of S lambda infty}; see also Lemma \ref{lem: esti gradient kernel 1}).
\end{remark}

\begin{example}\label{ex: norm xi power gamma}
\upshape
Let $\gamma>0$ and $\kappa>0$ be given
and let $k:[0,\infty)\to [0,\infty)$ be a nonnegative bounded measurable function
such that $k$ is bounded by $M\ge0$.
Consider the measurable function $\psi(t,\xi)=-(\kappa+k(t))|\xi|^{\gamma}$
for $t\ge0$ and $\xi\in\mathbb{R}^{d}$.
We show that $\psi$ satisfies the conditions \textbf{(S1)} and \textbf{(S2)}.
For any $t\ge0$ and $\xi\in\mathbb{R}^{d}$, it holds that
\begin{align*}
{\rm Re}[\psi(t,\xi)]=-(\kappa+k(t))|\xi|^{\gamma}\le -\kappa|\xi|^{\gamma},
\end{align*}
and so $\psi$ satisfies \textbf{(S1)}.
By direct computation, we can see that for any $c\in\mathbb{R}\setminus \{0\}$ and any multi index $\alpha\in \mathbb{N}_{0}^{d}$,
\begin{equation}\label{eqn:esti of deri of norm of xi power a}
|\partial_{\xi}^{\alpha}|\xi|^{c}|\leq K_{c,\alpha}|\xi|^{c-|\alpha|}
\end{equation}
for some $K_{c,\alpha}>0$.
Hence, by taking any natural number $N$, we obtain that for any $\alpha$ with $|\alpha|\leq N$,
\begin{align*}
|\partial_{\xi}^{\alpha}\psi(t,\xi)|
=|-(\kappa+k(t))\partial_{\xi}^{\alpha}|\xi|^{\gamma}|
\leq (\kappa+M)K_{\gamma,\alpha}|\xi|^{\gamma-|\alpha|},
\end{align*}
which implies that $\psi$ satisfies \textbf{(S2)}.
\end{example}

Let $\psi\in \mathfrak{S}$ be a symbol.
Then from \textbf{(S1)}, we can easily see that $\exp\left(\int_{s}^{t}\psi(r,\cdot)dr\right)\in L^1(\mathbb{R}^{d})$ for each $t>s$.
Define
\begin{align*}
p_{\psi}(t,s,x)=\mathcal{F}^{-1}\left(\exp\left(\int_{s}^{t}\psi(r,\cdot)dr\right)\right)(x),\quad t>s,\quad x\in\mathbb{R}^d.
\end{align*}

\begin{corollary}[\cite{I. Kim S. Lim K.-H. Kim 2016}, Lemma 5.2]\label{coro:Int of F of S}
Let $\psi$ be a measurable complex-valued function defined on $\mathbb{R}\times \mathbb{R}^{d}$
satisfying that there exist positive constants $\kappa,\gamma,\mu>0$ such that for almost all $t\in\mathbb{R}$ and $\xi\in\mathbb{R}^{d}$,
\begin{align}\label{eqn: symbol 1-1}
\mathrm{Re}[\psi(t,\xi)]\leq -\kappa |\xi|^{\gamma},
\end{align}
and for any multi-index $\alpha=(\alpha_{1},\cdots,\alpha_{d})$ with $|\alpha|:=\alpha_{1}+\cdots+\alpha\leq \lfloor\frac{d}{2}\rfloor+1$,
\begin{equation}\label{eqn: symbol 2-1}
      |\partial_{\xi}^{\alpha}\psi(t,\xi)|\leq \mu|\xi|^{\gamma-|\alpha|}
\end{equation}
for almost all $t\in\mathbb{R}$ and $\xi\in\mathbb{R}^{d}\setminus\widetilde{\boldsymbol{0}}$.
Put
\[
p_{\psi}(t,s,x):=\mathcal{F}^{-1}\left(\exp\left(\int_{s}^{t}\psi(r,\xi)dr\right)\right)(x),\quad x\in\mathbb{R}^{d}.
\]
Then we have $p_{\psi}(t,s,\cdot)\in L^{1}(\mathbb{R}^{d})$ for all $t>s$.
\end{corollary}

\begin{proof}
For any multi-index $\alpha$ with $|\alpha|\leq \lfloor\frac{d}{2}\rfloor+1$,
by direct computation using \eqref{eqn: symbol 1-1} and \eqref{eqn: symbol 2-1},
we can see that
\begin{equation}
\left|\partial_{\xi}^{\alpha}\exp\left(\int_{s}^{t}\psi(r,\xi)dr\right)\right|
\leq D(t-s)|\xi|^{\gamma-|\alpha|}e^{-c(t-s)|\xi|^{\gamma}}\label{eqn: psi 2 ineq-1}
\end{equation}
for some constant $D\ge0$.
Now by taking $\epsilon<\frac{\gamma}{4}\wedge \frac{1}{4}$, $\eta=\frac{d+\epsilon}{2}$,
$h(\xi)=\exp\left(\int_{s}^{t}\psi(r,\xi)dr\right)$ for each $t>s$, and $f(u)=e^{-c(t-s)u^{\gamma}}$,
we can see that the assumptions in Theorem \ref{thm:ph is in L1} hold,
and so we prove that
$p_{\psi}(t,s,\cdot)$ is integrable.
\end{proof}

Let $\psi\in \mathfrak{S}$ be a symbol.
From the conditions \textbf{(S1)} and \textbf{(S2)}, by applying Corollary \ref{coro:Int of F of S},
we see that $p_{\psi}(t,s,\cdot)\in L^{1}(\mathbb{R}^d)$ for each $t>s$.
For $f\in L^{p}(\mathbb{R}^{d})$, define
\begin{equation}\label{eqn: mathcal T}
\mathcal{T}_{\psi}(t,s)f(x)=p_{\psi}(t,s,\cdot)*f(x)=\int_{\mathbb{R}^{d}}p_{\psi}(t,s,x-y)f(y)dy.
\end{equation}
Then since $p_{\psi}(t,s,\cdot)\in L^{1}(\mathbb{R}^d)$ for each $t>s$,
by applying the Young's convolution inequality,
we can easily see that $\mathcal{T}_{\psi}(t,s)f\in L^{p}(\mathbb{R}^{d})$ for all $t>s$.
Also, the family $\{\mathcal{T}_{\psi}(t,s)\}_{t\geq s\geq 0}$ is an evolution system. Indeed, for any $t\geq r\geq s$ and $f\in L^{p}(\mathbb{R}^{d})$,
\begin{align*}
\mathcal{T}_{\psi}(t,r)\mathcal{T}_{\psi}(r,s)f(x)
&=p_{\psi}(t,r,\cdot)*\left(p_{\psi}(r,s,\cdot)*f\right)(x)\\
&=\mathcal{F}^{-1}\left(\exp\left(\int_{r}^{t}\psi(u,\xi)du\right)\exp\left(\int_{s}^{r}\psi(u,\xi)du\right)\mathcal{F}f(\xi)\right)(x)\\
&=p_{\psi}(t,s,\cdot)*f(x)\\
&=\mathcal{T}_{\psi}(t,s)f(x).
\end{align*}

\section{H\"ormander's Condition for Kernels}\label{sec:Estimates}

From now on, throughout this paper, we consider the pair $(\psi_{1},\psi_{2})\in \mathfrak{S}^{2}$ of symbols.
Then from the definition of the class $\mathfrak{S}$, for each $i=1,2$,
there exist positive constants $\kappa_{i}:=\kappa_{\psi_{i}}$, $\mu_{i}:=\mu_{\psi_{i}}$,
$\gamma_{i}:=\gamma_{\psi_{i}}$ and $N_{i}:=N_{\psi_{i}}$ with $N_{i}\ge \lfloor\frac{d}{2}\rfloor +1$ such that
\textbf{(S1)} and \textbf{(S2)} hold.
In this case, for notational convenience, we write
\begin{align*}
\boldsymbol{\kappa}:=(\kappa_1,\kappa_2),\quad
\boldsymbol{\mu}:=(\mu_{1},\mu_{2}),\quad
\boldsymbol{\gamma}:=(\gamma_{1},\gamma_{2}),\quad
\boldsymbol{N}:=(N_{1},N_{2}).
\end{align*}

Our approach to proving the Littlewood-Paley type inequality given in \eqref{eqn: LP ineq for q}
is to apply the Calder\'{o}n-Zygmund theorem for vector valued functions.
The left-hand side of inequality \eqref{eqn: LP ineq for q} can be expressed as follows:
\begin{equation}\label{eqn: another expression of LHS of LP ineq}
\int_{\mathbb{R}^{d}}\left(\int_{s}^{s+a}(t-s)^{\frac{q\gamma_{1}}{\gamma_{2}}-1}
           |\left(L_{\psi_{1}}(l)p_{\psi_{2}}(t,s,\cdot)\right)*f(x)|^{q}dt\right)^{\frac{p}{q}}dx.
\end{equation}
In fact, it holds that
\begin{align}
L_{\psi_{1}}(l)\mathcal{T}_{\psi_{2}}(t,s)f(x)
&=L_{\psi_{1}}(l)\left(p_{\psi_{2}}(t,s,\cdot)*f\right)(x)\nonumber\\
&=\left(L_{\psi_{1}}(l)p_{\psi_{2}}(t,s,\cdot)\right)*f(x).\label{eqn: kernel}
\end{align}
Then the inequality \eqref{eqn: LP ineq for q} means that
the convolution operator $Tf=L_{\psi_{1}}(l)p_{\psi_{2}}(\cdot,s,\cdot)*f$ is bounded
from $L^{p}(\mathbb{R}^{d})$ into $L^{p}(\mathbb{R}^{d};V)$
(under certain conditions; see Theorem \ref{thm: Lp boundedness of S lambda infty}),
where $V=L^{q}((s,s+a),(t-s)^{\frac{q\gamma_{1}}{\gamma_{2}}-1}dt)$.
For our purpose, we have to prove the H\"ormander's condition for kernels
and then we first estimate the gradient of the kernel $L_{\psi_{1}}(l)p_{\psi_{2}}(\cdot,s,\cdot)$.

\begin{lemma}\label{lem: esti of partial psi 1 psi 2}
Let $(\psi_{1},\psi_{2})\in \mathfrak{S}^{2}$
and let $\alpha\in\mathbb{N}_{0}^{d}$ be a multi-index with $|\alpha|\leq \min\{N_{1},\, N_{2}\}$.
Then there exist constants $c,C>0$ depending on $ \alpha,\,\boldsymbol{\kappa}$ and $\boldsymbol{\mu}$
such that for any $l\geq 0$ and $t>s$,
\begin{align*}
\left|\partial_{\xi}^{\alpha}\left[\psi_{1}(l,\xi)\exp\left(\int_{s}^{t}\psi_{2}(r,\xi)dr\right)\right]\right|
\leq C|\xi|^{\gamma_{1}-|\alpha|}e^{-c(t-s)|\xi|^{\gamma_{2}}}.
\end{align*}
\end{lemma}
\begin{proof}
By applying the formula given in \eqref{eqn: psi 2 ineq-1}
and the fact that $xe^{-x}\leq 1$ for any $x>0$,
for any multi index $\beta\in\mathbb{N}_{0}^{d}$ with $|\beta|\leq N_{2}$, we obtain that
\begin{align}\label{eqn:psi 2 ineq-2}
\left|\partial_{\xi}^{\beta}\exp\left(\int_{s}^{t}\psi_{2}(r,\xi)dr\right)\right|
&\leq D(t-s)|\xi|^{\gamma_{2}-|\beta|}e^{-c_{1}(t-s)|\xi|^{\gamma_{2}}}\nonumber\\
&=D'|\xi|^{-|\beta|}e^{-c(t-s)|\xi|^{\gamma_{2}}},
\end{align}
for some constant $D,D',c_{1},c>0$.
By applying the Leibniz rule, \textbf{(S2)}, \eqref{eqn:psi 2 ineq-2}, we obtain that
\begin{align*}
&\left|\partial_{\xi}^{\alpha}\left[\psi_{1}(l,\xi)
     \exp\left(\int_{s}^{t}\psi_{2}(r,\xi)dr\right)\right]\right|\\
&\leq\sum_{\beta\leq \alpha}\binom{\alpha}{\beta}\left|\partial_{\xi}^{\alpha-\beta}\psi_{1}(l,\xi)\right|
\left|\partial_{\xi}^{\beta}\exp\left(\int_{s}^{t}\psi_{2}(r,\xi)dr\right)\right| \\
&\leq \sum_{\beta\leq \alpha}\binom{\alpha}{\beta}\mu_{1}|\xi|^{\gamma_{1}-|\alpha-\beta|}
           D'|\xi|^{-|\beta|}e^{-c(t-s)|\xi|^{\gamma_{2}}},
\end{align*}
which implies the desired assertion.
\end{proof}

For $l\geq 0$, $t>s$ and $x\in\mathbb{R}^{d}$, we put
\begin{align*}
F_{\psi_{1},\psi_{2}}(l,t,s,x)
 &=\sum_{j=1}^{d}\left|\mathcal{F}^{-1}\left(\xi_{j}\psi_{1}(l,\xi)\exp\left(\int_{s}^{t}\psi_{2}(r,\xi)dr\right)\right)(x)\right|.
\end{align*}

\begin{lemma}\label{lem: esti gradient kernel 1}
Let $(\psi_{1},\psi_{2})\in \mathfrak{S}^{2}$ and let $N=\min\{N_{1},N_{2}\}$.
Suppose that $N>d+1+\lfloor\gamma_{1}\rfloor$.
Then there exists a constant $C>0$ depending on $\mu_{1},\gamma_{1},d$ and $N$ such that
for any $l\geq 0$, $t>s$ and $x\in\mathbb{R}^{d}\setminus \{0\}$,
\begin{align*}
F_{\psi_{1},\psi_{2}}(l,t,s,x)\leq C|x|^{-(d+1+\gamma_{1})}.
\end{align*}
\end{lemma}

\begin{proof}
We use the arguments used in \cite{Miao 2008}.
Define a function $\rho\in C_{c}^{\infty}(\mathbb{R}^{d})$ by
\[
\rho(\xi)=\left\{
            \begin{array}{ll}
              1, & \hbox{$|\xi|\leq 1$,} \\
              0, & \hbox{$|\xi|> 2$,}
            \end{array}
          \right.\qquad \xi\in\mathbb{R}^{d}.
\]
Let $\lambda>0$ be given. Note that we have
\begin{equation}\label{eqn: rho lambda}
\rho\left(\frac{\xi}{\lambda}\right)=\left\{
            \begin{array}{ll}
              1, & \hbox{$|\xi|\leq \lambda$,} \\
              0, & \hbox{$|\xi|> 2\lambda$,}
            \end{array}
          \right.\qquad
1-\rho\left(\frac{\xi}{\lambda}\right)=\left\{
            \begin{array}{ll}
              0, & \hbox{$|\xi|\leq \lambda$,} \\
              1, & \hbox{$|\xi|> 2\lambda$.}
            \end{array}
          \right.
\end{equation}
For each $j=1,2,\cdots,d$, we
put
\[
G_{j}(\xi)=\xi_{j}\psi_{1}(l,\xi)
       \exp\left(\int_{s}^{t}\psi_{2}(r,\xi)dr\right).
\]
Let $x\in\mathbb{R}^{d}\setminus \{0\}$ be given.
Then by applying \eqref{eqn: rho lambda}, we have
\begin{align}
\mathcal{F}^{-1}\left(G_{j}(\xi)\right)(x)
&=\int_{\mathbb{R}^d}e^{ix\cdot\xi}G_{j}(\xi)\rho\left(\frac{\xi}{\lambda}\right)d\xi
     +\int_{\mathbb{R}^d}e^{ix\cdot\xi}G_{j}(\xi)\left(1-\rho\left(\frac{\xi}{\lambda}\right)\right)d\xi\nonumber\\
&=I_{1}+I_{2},\label{eqn: I1 I2}
\end{align}
where
\begin{align*}
I_{1}=\int_{|\xi|\leq 2\lambda}e^{ix\cdot\xi}G_{j}(\xi)\rho\left(\frac{\xi}{\lambda}\right)d\xi,
\quad
I_{2}=\int_{|\xi|>\lambda}e^{ix\cdot\xi}G_{j}(\xi)\left(1-\rho\left(\frac{\xi}{\lambda}\right)\right)d\xi.
\end{align*}
For $I_{1}$, since $|\rho(\xi)|\leq 1$ for all $\xi\in\mathbb{R}^{d}$,
by the Assumptions \textbf{(S1)} and \textbf{(S2)},
and applying the change of variables to polar coordinates,
we obtain that
\begin{align}
|I_{1}|
&\leq\int_{|\xi|\leq 2\lambda}|\xi_{j}\psi_{1}(l,\xi)|
      \left|\exp\left(\int_{s}^{t}\psi_{2}(r,\xi)dr\right)\right|d\xi\nonumber\\
&\le \mu_{1}\int_{|\xi|\leq 2\lambda}|\xi_{j}||\xi|^{\gamma_{1}}d\xi
\leq \mu_{1}\frac{2\pi^{\frac{d}{2}}}{\Gamma(\frac{d}{2})}\int_{0}^{2\lambda}r^{\gamma_{1}+1+d-1}dr\nonumber\\
&=\mu_{1}\frac{2\pi^{\frac{d}{2}}}{\Gamma(\frac{d}{2})}(2\lambda)^{\gamma_{1}+1+d}.\label{eqn: Gi I1}
\end{align}
Now we estimate $I_{2}$. For a differentiable function $f$ on $\mathbb{R}^{d}$, define
\[
\mathcal{D}f(\xi)=\frac{x\cdot\nabla_{\xi}f(\xi)}{i|x|^2},\qquad \xi\in\mathbb{R}^{d}.
\]
Then it is obvious that $\mathcal{D}e^{ix\cdot\xi}=e^{i x\cdot \xi}$.
Therefore, we obtain that
\begin{align}
I_{2}&=\int_{|\xi|>\lambda}e^{ix\cdot\xi}G_{j}(\xi)\left(1-\rho\left(\frac{\xi}{\lambda}\right)\right)d\xi
=\int_{|\xi|>\lambda}\left(\mathcal{D}^{N}e^{ix\cdot\xi}\right)G_{j}(\xi)\left(1-\rho\left(\frac{\xi}{\lambda}\right)\right)d\xi\nonumber\\
&=\int_{|\xi|>\lambda}e^{ix\cdot\xi}
     \left[(\mathcal{D}^*)^{N}\left(G_{j}(\xi)\left(1-\rho\left(\frac{\xi}{\lambda}\right)\right)\right)\right]d\xi\nonumber\\
&=I_{21}+I_{22},\label{eqn: decomp of I2}
\end{align}
where $\mathcal{D}^{*}=-\frac{x\cdot\nabla_{\xi}}{i|x|^2}$ and
\begin{align*}
I_{21}
=\int_{\lambda< |\xi|\leq 2\lambda}
    e^{ix\cdot\xi}(\mathcal{D}^{*})^{N}\left(G_{j}(\xi)\left(1-\rho\left(\frac{\xi}{\lambda}\right)\right)\right)d\xi,
~~
I_{22}=\int_{|\xi|> 2\lambda}e^{ix\cdot\xi}(\mathcal{D}^{*})^{N}G_{j}(\xi)d\xi.
\end{align*}
We estimate $I_{21}$. Note that
\[
(\mathcal{D}^{*})^{N}
=\frac{(-1)^{N}}{i^{N}|x|^{2N}}\sum_{|\alpha|=N}x^{\alpha}\partial_{\xi}^{\alpha},
\]
where $x^{\alpha}=x_{1}^{\alpha_{1}}\cdots x_{d}^{\alpha_{d}}$ for $x=(x_{1},\cdots,x_{d})\in\mathbb{R}^{d}$
and a multi-index $\alpha=(\alpha_{1},\cdots,\alpha_{d})\in\mathbb{N}_{0}^{d}$.
Therefore, by applying the Leibniz rule, we obtain that
\begin{align*}
(\mathcal{D}^{*})^{N}\left(G_{j}(\xi)\left(1-\rho\left(\frac{\xi}{\lambda}\right)\right)\right)
&=\frac{(-1)^{N}}{i^{N}|x|^{2N}}\sum_{|\alpha|=N}x^{\alpha}\partial_{\xi}^{\alpha}
    \left(G_{j}(\xi)\left(1-\rho\left(\frac{\xi}{\lambda}\right)\right)\right)\\
&=\frac{(-1)^{N}}{i^{N}|x|^{2N}}\sum_{|\alpha|=N}x^{\alpha}\sum_{\beta\leq \alpha}\binom{\alpha}{\beta}\\
&\qquad\times    \left(\partial_{\xi}^{\alpha-\beta}G_{j}(\xi)\right)
       \left(\partial_{\xi}^{\beta}\left(1-\rho\left(\frac{\xi}{\lambda}\right)\right)\right).
\end{align*}
Using Lemma \ref{lem: esti of partial psi 1 psi 2}, we have
\begin{align}\label{eqn: partial alpha Gi}
|\partial_{\xi}^{\alpha-\beta}G_{j}(\xi)|
&\leq C_{\alpha-\beta}^{(1)}|\xi|^{\gamma_{1}+1-|\alpha-\beta|}e^{-c(t-s)|\xi|^{\gamma_{2}}}
\leq C_{\alpha-\beta}^{(1)}|\xi|^{\gamma_{1}+1-|\alpha-\beta|}.
\end{align}
Also, since $\lambda^{-|\beta|}\leq 2^{|\beta|}|\xi|^{-|\beta|}$ for $|\xi|\leq 2\lambda$, we have
\begin{equation}\label{eqn: partial 1-rho}
\left|\partial_{\xi}^{\beta}\left(1-\rho\left(\frac{\xi}{\lambda}\right)\right)\right|
\leq C_{\beta}^{(2)}\lambda^{-|\beta|}
\le C_{\beta}^{(2)}2^{|\beta|}|\xi|^{-|\beta|}.
\end{equation}
Therefore, by applying \eqref{eqn: partial alpha Gi} and \eqref{eqn: partial 1-rho}, we obtain that
\begin{align}
&\left|(\mathcal{D}^{*})^{N}\left(G_{j}(\xi)\left(1-\rho\left(\frac{\xi}{\lambda}\right)\right)\right)\right|\nonumber\\
&\leq \frac{1}{|x|^{2N}}\sum_{|\alpha|=N}|x^{\alpha}|\sum_{\beta\leq \alpha}\binom{\alpha}{\beta}\left|\partial_{\xi}^{\alpha-\beta}G_{j}(\xi)\right|
      \left|\partial_{\xi}^{\beta}\left(1-\rho\left(\frac{\xi}{\lambda}\right)\right)\right|\nonumber\\
&\leq\frac{1}{|x|^{N}}\sum_{|\alpha|=N}\sum_{\beta\leq \alpha}
     \binom{\alpha}{\beta}C_{\alpha-\beta}^{(1)}C_{\beta}^{(2)} 2^{|\beta|} |\xi|^{\gamma_{1}+1-|\alpha-\beta|} |\xi|^{-|\beta|}\nonumber\\
&=C^{(3)}\frac{1}{|x|^{N}}|\xi|^{\gamma_{1}+1-N}\label{eqn: L*N Gi 1-rho}
\end{align}
for some $C^{(3)}>0$.
By \eqref{eqn: L*N Gi 1-rho}, changing variables to the polar coordinates, and the assumption $N>d+1+\lfloor\gamma_{1}\rfloor$, we obtain that
\begin{align*}
|I_{21}|
&\leq \int_{\lambda\leq |\xi|\leq 2\lambda}
           \left|(\mathcal{D}^{*})^{N}\left(G_{j}(\xi)\left(1-\rho(\frac{\xi}{\delta})\right)\right)\right|d\xi
\le C^{(3)}\frac{1}{|x|^{N}}\int_{\lambda\leq |\xi|\leq 2\lambda}|\xi|^{\gamma_{1}+1-N}d\xi\\
&=C^{(3)}\frac{1}{|x|^{N}}\frac{2\pi^{\frac{d}{2}}}{\Gamma(\frac{d}{2})}
\int_{\lambda}^{2\lambda}r^{\gamma_{1}+1-N+d-1}dr\\
&=C^{(4)}\frac{1}{|x|^{N}}\lambda^{\gamma_{1}+1-N+d}
\end{align*}
for some $C^{(4)}>0$.

We now estimate $I_{22}$. By Lemma \ref{lem: esti of partial psi 1 psi 2},
since $e^{-c(t-s)|\xi|^{\gamma_{2}}}\leq 1$, we obtain that
\begin{align*}
\left|((\mathcal{D}^{*})^{N}G_{j})(\xi)\right|
&\leq \frac{1}{|x|^{2N}}\sum_{|\alpha|=N}|x^{\alpha}
|\left|\partial_{\xi}^{\alpha}G_{j}(\xi)\right|\nonumber\\
&\leq \frac{1}{|x|^{2N}}\sum_{|\alpha|=N}|x|^{|\alpha|}C_{\alpha}^{(5)}e^{-c(t-s)|\xi|^{\gamma_{2}}}|\xi|^{\gamma_{1}+1-|\alpha|}\nonumber\\
&=C^{(6)}\frac{1}{|x|^{N}}|\xi|^{\gamma_{1}+1-N}
\end{align*}
for some $C^{(6)}>0$.
By the assumption $N>d+1+\lfloor\gamma_{1}\rfloor$,
we obtain that
\begin{align*}
|I_{22}|&\leq \int_{|\xi|\geq 2\lambda}\left|(\mathcal{D}^{*})^{N}\left(G_{j}(\xi)\right)\right|d\xi
\leq C^{(6)}\frac{1}{|x|^{N}}\int_{|\xi|\geq 2\lambda}|\xi|^{\gamma_{1}+1-N}d\xi\\
&\leq C^{(6)}\frac{1}{|x|^{N}}
\frac{2\pi^{\frac{d}{2}}}{\Gamma(\frac{d}{2})}\int_{2\lambda}^{\infty}r^{\gamma_{1}+1-N+d-1}dr\\
&=C^{(7)}\frac{1}{|x|^{N}}\lambda^{\gamma_{1}+1-N+d}
\end{align*}
for some $C^{(7)}>0$.
Therefore, by \eqref{eqn: I1 I2}, \eqref{eqn: Gi I1} and \eqref{eqn: decomp of I2}, we have for any $i=1,2,\cdots,d,$
\begin{equation}
|\mathcal{F}^{-1}(G_{j}(\xi))(x)|
\leq \mu_{1}\frac{2\pi^{\frac{d}{2}}}{\Gamma(\frac{d}{2})}2^{\gamma_{1}+1+d}\lambda^{\gamma_{1}+d+1}
+C^{(8)}\frac{1}{|x|^{N}}\lambda^{\gamma_{1}+d+1-N}\label{eqn: IFT of Gi less lambda}
\end{equation}
for some $C^{(8)}>0$.
By taking $\lambda=|x|^{-1}$ in \eqref{eqn: IFT of Gi less lambda}, we get
\begin{align*}
F_{\psi_{1},\psi_{2}}(l,t,s,x)
&=\sum_{j=1}^{d}|\mathcal{F}^{-1}(G_{j}(\xi))(x)|
\leq C|x|^{-(d+1+\gamma_{1})}
\end{align*}
for some $C>0$.
The proof is complete.
\end{proof}

For a real-valued differentiable function $f$ on $\mathbb{R}^{d}$, we denote the gradient of $f$  by $\nabla f$.

\begin{lemma}\label{lem: esti of grad K}
Let $(\psi_{1},\psi_{2})\in\mathfrak{S}^{2}$ and let $N=\min\{N_{1},N_{2}\}$. Assume that $N>d+1+\lfloor\gamma_{1}\rfloor$. Then there exists a constant $C>0$ depending on $\mu_{1},\boldsymbol{\gamma},\kappa_{2}, d$ and $N$ such that
for any $l\geq 0$, $t>s$ and $x\in\mathbb{R}^{d}\setminus\{0\}$,
\begin{align}
|\nabla L_{\psi_{1}}(l)p_{\psi_{2}}(t,s,x)|
\leq C\left(|x|^{-(d+1+\gamma_{1})}\wedge (t-s)^{-\frac{d+1+\gamma_{1}}{\gamma_{2}}}\right).
  \label{eqn: bound of Dx K}
\end{align}
\end{lemma}

\begin{proof}
For $i=1,2,\cdots,d$, we have
\begin{align*}
\partial_{x_{i}}L_{\psi_{1}}(l)p_{\psi_{2}}(t,s,x)
=\mathcal{F}^{-1}\left(\xi_{i}\psi_{1}(l,\xi)\exp\left(\int_{s}^{t}\psi_{2}(r,\xi)dr\right)\right)(x).
\end{align*}
Since $(a+b)^{\frac{1}{p}}\leq a^{\frac{1}{p}}+b^{\frac{1}{p}}$ for any $p\geq 1$, by Lemma \ref{lem: esti gradient kernel 1},
we obtain that
\begin{align}
|\nabla L_{\psi_{1}}(l)p_{\psi_{2}}(t,s,x)|
&=\left(\sum_{i=1}^{d}|\partial_{x_{i}}L_{\psi_{1}}(l)p_{\psi_{2}}(t,s,x)|^{2}\right)^{\frac{1}{2}}
\leq \sum_{i=1}^{d}|\partial_{x_{i}}L_{\psi_{1}}(l)p_{\psi_{2}}(t,s,x)|\nonumber\\
&=\sum_{i=1}^{d}\left|\mathcal{F}^{-1}\left(\xi_{i}
    \psi_{1}(l,\xi)\exp\left(\int_{s}^{t}\psi_{2}(r,\xi)dr\right)\right)(x)\right|\nonumber\\
&= F_{\psi_{1},\psi_{2}}(l,t,s,x)\nonumber\\
&\leq C^{(1)}|x|^{-(d+1+\gamma_{1})}.\label{eqn: esti nabla K 1}
\end{align}
On the other hand, by Assumptions \textbf{(S1)} and \textbf{(S2)} and changing variables by $(t-s)^{\frac{1}{\gamma_{2}}}\xi\rightarrow y$, we obtain that for any $x\in\mathbb{R}^{d}$,
\begin{align}
F_{\psi_{1},\psi_{2}}(l,t,s,x)
&=\sum_{i=1}^{d}\left|\mathcal{F}^{-1}\left(\xi_{i}\psi_{1}(l,\xi)
    \exp\left(\int_{s}^{t}\psi_{2}(r,\xi)dr\right)\right)(x)\right|\nonumber\\
&\leq \sum_{i=1}^{d} \int_{\mathbb{R}^{d}}|\xi_{i}||\psi_{1}(l,\xi)|
     \left|\exp\left(\int_{s}^{t}\psi_{2}(r,\xi)dr\right)\right|d\xi\nonumber\\
&\leq \sum_{i=1}^{d}\int_{\mathbb{R}^{d}}\mu_{1}
|\xi|^{\gamma_{1} +1}e^{-\kappa_{2}(t-s)|\xi|^{\gamma_{2}}}d\xi\nonumber\\
&=\sum_{i=1}^{d}\int_{\mathbb{R}^{d}}\mu_{1}
|(t-s)^{-\frac{1}{\gamma_{2}}}y|^{\gamma_{1}+1}e^{-\kappa_{2}|y|^{\gamma_{2}}}(t-s)^{-\frac{d}{\gamma_{2}}}dy\nonumber\\
&=C^{(2)}(t-s)^{-\frac{d+1+\gamma_{1}}{\gamma_{2}}}\label{eqn: esti nabla K 2}
\end{align}
for some $C^{(2)}>0$.
Hence, by combining \eqref{eqn: esti nabla K 1} and \eqref{eqn: esti nabla K 2},
we have the inequality given in \eqref{eqn: bound of Dx K}.
\end{proof}

In the following proposition, we prove that
the function $K(x)=L_{\psi_{1}}(l)p_{\psi_{2}}(\cdot,s,x)$ (for $x\in\mathbb{R}^{d}$)
satisfies the H\"ormander's condition given in \eqref{eqn:Hormander's condition} with
$V=L^{q}((s,s+a), (t-s)^{\frac{q\gamma_{1}}{\gamma_{2}}-1}dt)$,
which will be used to apply Corollary \ref{cor: boundedness of singular integral}
in order to prove Theorem \ref{thm: convolution is bdd on Lp}.

\begin{proposition}\label{prop: estimation of int of Kt(x-y)-Kt(x) over outside ball}
Let $(\psi_{1},\psi_{2})\in\mathfrak{S}^{2}$.
Let $q\geq 2$, $s\geq 0$, $0<a\leq \infty$ and let $V=L^{q}((s,s+a),(t-s)^{\frac{q\gamma_{1}}{\gamma_{2}}-1}dt)$.
Suppose that $N_1,N_2>d+1+\lfloor\gamma_{1}\rfloor$.
Then there exists a constant $A>0$ such that
\begin{align}\label{eqn:Hormander's condition}
\int_{|x|\geq 2|y|}\left\|K(x-y)-K(x)\right\|_{V}dx
\leq A
\end{align}
for any $y\neq 0$, where $K(x):=K(\cdot,x):=L_{\psi_{1}}(l)p_{\psi_{2}}(\cdot,s,x)$ (for $x\in\mathbb{R}^{d}$).
\end{proposition}

\begin{proof}
By applying the mean value theorem and Cauchy-Schwarz inequality,
for any $a>0$, we obtain that
\begin{align}
\left\|K(x-y)-K(x)\right\|_{V}^{q}
&=\int_{s}^{s+a}|K(t,x-y)-K(t,x)|^{q}
      (t-s)^{\frac{q\gamma_{1}}{\gamma_{2}}-1}dt\nonumber\\
&=\int_{s}^{s+a}|\nabla K(t,x-\theta y)\cdot (-y)|^{q}
      (t-s)^{\frac{q\gamma_{1}}{\gamma_{2}}-1}dt\nonumber\\
&\leq |y|^{q}\int_{s}^{s+a}|\nabla K(t,x-\theta y)|^{q}(t-s)^{\frac{q\gamma_{1}}{\gamma_{2}}-1}dt\nonumber\\
&=|y|^{q}\int_{s}^{s+|x-\theta y|^{\gamma_{2}}}|\nabla K(t,x-\theta y)|^{q}(t-s)^{\frac{q\gamma_{1}}{\gamma_{2}}-1}dt\nonumber\\
&\qquad+|y|^{q}\int_{s+|x-\theta y|^{\gamma_{2}}}^{\infty}|\nabla K(t,x-\theta y)|^{q}(t-s)^{\frac{q\gamma_{1}}{\gamma_{2}}-1}dt
  \label{eqn: norm of K(x-y)-K(x)}
\end{align}
for some $\theta\in (0,1)$.
By assumption $N_{1},N_{2}>d+1+\lfloor\gamma_{1}\rfloor$,
and then by Lemma \ref{lem: esti of grad K}, there exists a constant $C_{1}>0$ such that
\begin{align*}
|\nabla K(t,x-\theta y)|
&=\left|\nabla L_{\psi_{1}}(l)p_{\psi_{2}}(t,s,x-\theta y)\right|\\
&\leq C_{1} \left(|x-\theta y|^{-(d+1+\gamma_{1})}\wedge (t-s)^{-\frac{d+1+\gamma_{1}}{\gamma_{2}}}\right),
\end{align*}
from which for $t-s<|x-\theta y|^{\gamma_{2}}$, it holds that
\begin{equation}\label{eqn: gradient of K 1}
|\nabla K(t,x-\theta y)|
\leq C_{1} |x-\theta y|^{-(d+1+\gamma_{1})}
\end{equation}
and for $t-s\geq |x-\theta y|^{\gamma_{2}}$,
\begin{equation}\label{eqn: gradient of K 2}
|\nabla K(t,x-\theta y)|
\leq C_{1} (t-s)^{-\frac{d+1+\gamma_{1}}{\gamma_{2}}}.
\end{equation}
Therefore, for any $a>0$,
by combining \eqref{eqn: norm of K(x-y)-K(x)}, \eqref{eqn: gradient of K 1} and \eqref{eqn: gradient of K 2},
we obtain that
\begin{align}\label{eqn: int 0 infty nabla K}
\left\|K(x-y)-K(x)\right\|_{V}^{q}
&\leq C_{1}^{q} |y|^{q}\int_{s}^{s+|x-\theta y|^{\gamma_{2}}}
      |x-\theta y|^{-q(d+1+\gamma_{1})}(t-s)^{\frac{q\gamma_{1}}{\gamma_{2}}-1} dt\nonumber\\
&\quad+C_{1}^{q} |y|^{q} \int_{s+|x-\theta y|^{\gamma_{2}}}^{\infty}
       (t-s)^{-\frac{q(d+1+\gamma_{1})}{\gamma_{2}}}(t-s)^{\frac{q\gamma_{1}}{\gamma_{2}}-1}dt\nonumber\\
&= \frac{C_{1}^{q}\gamma_{2} }{q\gamma_{1}}|y|^{q} |x-\theta y|^{-q(d+1)}
   +\frac{C_{1}^{q}\gamma_{2}}{q(d+1)}|y|^{q}|x-\theta y|^{-q(d+1)}\nonumber\\
&=C_{2}|y|^{q}|x-\theta y|^{-q(d+1)},\quad C_{2}=\frac{C_{1}^{q}\gamma_{2}}{q}\left(\frac{1}{\gamma_{1} }+\frac{1}{(d+1)}\right).
\end{align}
By combining \eqref{eqn: norm of K(x-y)-K(x)} and \eqref{eqn: int 0 infty nabla K}
and changing variable by $x-\theta y=z$, we obtain that
\begin{align}
\int_{|x|\geq 2|y|}\left\|K(x-y)-K(x)\right\|_{V}dx
&\leq C_{2}^{\frac{1}{q}}|y|\int_{|x|\geq 2|y|}|x-\theta y|^{-d-1}dx\nonumber\\
&=C_{2}^{\frac{1}{q}}|y|\int_{|z|\geq (2-\theta)|y|}|z|^{-d-1}dz\nonumber\\
&\le C_{2}^{\frac{1}{q}}|y|\int_{|z|\geq |y|}|z|^{-d-1}dz.
    \label{eqn: integral of norm of K(x-y)-K(x)}
\end{align}
Let $\partial B_{1}$ be the unit sphere in $\mathbb{R}^{d}$ and let $\sigma$ be the surface measure on $\partial B_{1}$.
By changing variable by the polar coordinate, we have
\begin{align}\label{eqn: int |x-theta y|^(-d-1)}
\int_{|z|\geq |y|}|z|^{-d-1}dz
=\sigma(\partial B_{1})\int_{|y|}^{\infty}\rho^{-d-1}\rho^{d-1}d\rho
=\frac{2\pi^{d/2}}{\Gamma\left(\frac{d}{2}\right)}\int_{|y|}^{\infty}\rho^{-2} d\rho
\leq\frac{2\pi^{d/2}}{\Gamma\left(\frac{d}{2}\right)}|y|^{-1}.
\end{align}
Therefore, by combining \eqref{eqn: integral of norm of K(x-y)-K(x)} and \eqref{eqn: int |x-theta y|^(-d-1)}, we have
\begin{align*}
\int_{|x|\geq 2|y|}\left\|K(x-y)-K(x)\right\|_{V}dx
\leq C_{2}^{\frac{1}{q}}\frac{2\pi^{d/2}}{\Gamma\left(\frac{d}{2}\right)}.
\end{align*}
By taking $A=C_{2}^{\frac{1}{q}}\frac{2\pi^{d/2}}{\Gamma\left(\frac{d}{2}\right)}$, we get the desired result.
\end{proof}

\section{Convolution Operators on Besov Space}\label{sec:Convolution Operators}

We recall the Littlewood-Paley decomposition (see \eqref{eqn: LP decomposition}).
It is well known that there exists a function $\Phi\in C_{c}^{\infty}(\mathbb{R}^{d})$
such that $\Phi$ is nonnegative and
\begin{align}
&\text{supp}\, \Phi =\left\{\xi\in\mathbb{R}^{d}\,\left|\, {1}/{2}\leq |\xi|\leq 2\right.\right\},\nonumber\\
&\sum_{j=-\infty}^{\infty}\Phi(2^{-j}\xi)=1\label{eqn: sum of phi is one}
\end{align}
(see, e.g., Lemma 6.1.7 in \cite{Bergh 1976}).
For each $f\in L^{2}(\mathbb{R}^{d})$ and $j\in\mathbb{Z}$, we put
\begin{align*}
\Delta_{j}f(x)=\mathcal{F}^{-1}(\Phi(2^{-j}\xi)\mathcal{F}f(\xi))(x)
\end{align*}
and
\[
S_{0}(f)(x)=\sum_{j=-\infty}^{0}\Delta_{j}f(x).
\]
Then from \eqref{eqn: sum of phi is one}, we have
\begin{equation}\label{eqn: LP decomposition}
f(x)=S_{0}(f)(x)+\sum_{j=1}^{\infty}\Delta_{j}f(x),
\end{equation}
which is referred to as the Littlewood-Paley decomposition of $f$.
The following property is known as the almost orthogonality: for any $i,j$ with $|i-j|\geq 2$,
\[
\Delta_{i}\Delta_{j}=0,
\]
from which, we have the following equality:
\begin{align}\label{eqn: pseudo orthogonal property}
\left(L_{\psi_{1}}(l)p_{\psi_{2}}(t,s,\cdot)\right)*f
&=\sum_{j=-\infty}^{1}\left(L_{\psi_{1}}(l)p_{\psi_{2},j}(t,s,\cdot)\right)*S_{0}(f)\nonumber\\
&\qquad   +\sum_{j=1}^{\infty}\sum_{j-1\leq i\leq j+1}\left(L_{\psi_{1}}(l)p_{\psi_2,i}(t,s,\cdot)\right)*f_{j},
\end{align}
where $f_{i}=\Delta_{i}f$ and $p_{\psi_{2},j}(t,s,x)=\Delta_{ j}p_{\psi_{2}}(t,s,x)$.

\begin{lemma}\label{lem: esti 1}
Let $(\psi_{1},\psi_{2})\in\mathfrak{S}^{2}$, $N=\min\{N_{1},N_{2}\}$ and let $\alpha\in\mathbb{N}_{0}^{d}$ be a multi index with $|\alpha|\leq N$. Then there exist positive constants $c,C$ depending on $\mu_{1},\boldsymbol{\gamma},\kappa_{2}, d$ and $N$ such that
for any $l\geq 0$, $j\in\mathbb{Z}$ and $t>s$,
\begin{align*}
\left|\partial_{\xi}^{\alpha}\left[\psi_{1}(l,2^{j}\xi)\Phi(\xi)\exp\left(\int_{s}^{t}\psi_{2}(r,2^{j}\xi)dr\right)\right]\right|
\leq C2^{j\gamma_{1}}|\xi|^{\gamma-|\alpha|}1_{(\frac{1}{2},2)}(|\xi|)e^{-c(t-s)|2^{j}\xi|^{\gamma_{2}}}.
\end{align*}
\end{lemma}

\begin{proof}
By applying Lemma \ref{lem: esti of partial psi 1 psi 2} and the chain rule, we have
\begin{align}
\left|\partial_{\xi}^{\alpha}\left[\psi_{1}(l,2^{j}\xi)\exp\left(\int_{s}^{t}\psi_{2}(r,2^{j}\xi)dr\right)\right]\right|
\leq C_{1}2^{j\gamma_{1}}|\xi|^{\gamma_{1} -|\alpha|}e^{-c_{1}(t-s)|2^{j}\xi|^{\gamma_{2}}}\label{eqn: first ineq}
\end{align}
for some constant $c_{1},C_{1}>0$.
Recall that the support of $\Phi(\xi)$ is $\{\xi\in\mathbb{R}^{d}\,|\, \frac{1}{2}\leq |\xi|\leq 2\}$ and
\begin{equation}
\sup_{\xi}|\partial_{\xi}^{\alpha}\Phi(\xi)|
\leq C_{\alpha}.\label{eqn: third ineq}
\end{equation}
Hence, by \eqref{eqn: first ineq}, \eqref{eqn: third ineq} and Leibniz rule, we have
\begin{align*}
\left|\partial_{\xi}^{\alpha}\left[\psi_{1}(l,2^{j}\xi)\Phi(\xi)\exp\left(\int_{s}^{t}\psi_{2}(r,2^{j}\xi)dr\right)\right]\right|
\leq C2^{j\gamma_{1}}|\xi|^{\gamma_{1}-|\alpha|}1_{(\frac{1}{2},2)}(|\xi|)e^{-c_{1}(t-s)|2^{j}\xi|^{\gamma_{2}}}
\end{align*}
for some constant $C>0$.
The proof is complete.
\end{proof}

\begin{lemma}\label{lem: boundedness of L psi of p j}
Let $(\psi_{1},\psi_{2})\in\mathfrak{S}^{2}$.
There exist constants $c,C>0$ depending on $ \boldsymbol{\kappa},\boldsymbol{\mu}$ and $d$ such that
for any $l\geq 0$, $j\in\mathbb{Z}$ and $t>s$,
\[
\|L_{\psi_{1}}(l)p_{\psi_{2},j}(t,s,\cdot)\|_{L^{1}(\mathbb{R}^{d})}
\leq C2^{j\gamma_{1}}e^{-c(t-s)2^{j\gamma_{2}}}.
\]
\end{lemma}

\begin{proof}
By changing variable, we have
\begin{align*}
L_{\psi_{1}}(l)p_{\psi_{2},j}(t,s,x)
&=\mathcal{F}^{-1}\left(\psi_{1}(l,\xi)\Phi(2^{-j}\xi)\exp\left(\int_{s}^{t}\psi_{2}(r,\xi)dr\right)\right)(x)\\
&=2^{jd}\mathcal{F}^{-1}\left(h(\xi)\right),
\quad h(\xi)=\psi_{1}(l,2^{j}\xi)\Phi(\xi)\exp\left(\int_{s}^{t}\psi_{2}(r,2^{j}\xi)dr\right).
\end{align*}
By taking $\epsilon<\frac{\gamma_{1}}{4}\wedge \frac{1}{4}$ and $\eta=\frac{d+\epsilon}{2}$,
we see that $2\lfloor\frac{\eta}{2}\rfloor \leq \lfloor\frac{d}{2}\rfloor$,
which implies that
\begin{align*}
2\lfloor\frac{\eta}{2}\rfloor
&\leq 1+2\lfloor\frac{\eta}{2}\rfloor
\leq \lfloor\frac{d}{2}\rfloor+1,
\end{align*}
and if $1\leq \eta-2\lfloor\frac{\eta}{2}\rfloor$, then it holds that $1+2\lfloor\frac{\eta}{2}\rfloor\leq \lfloor\frac{d}{2}\rfloor$.
By assumption of $\mathfrak{S}$, $N_{1},N_{2}\geq \lfloor\frac{d}{2}\rfloor+1$.
Therefore, by Lemma \ref{lem: esti 1} and direct computation, we obtain that
\begin{align*}
\left|(-\Delta)^{\lfloor\frac{\eta}{2}\rfloor}h(\xi)\right|
&\leq C_{1}2^{j\gamma_{1}}|\xi|^{\gamma_{1}-2\lfloor\frac{\eta}{2}\rfloor}
1_{(\frac{1}{2},2)}(|\xi|)e^{-c_{1}(t-s)|2^{j}\xi|^{\gamma_{2}}},\\
\left|\nabla(-\Delta)^{\lfloor\frac{\eta}{2}\rfloor}h(\xi)\right|
&\leq C_{2}2^{j\gamma_{1}}|\xi|^{\gamma_{1}-2\lfloor\frac{\eta}{2}\rfloor-1}
1_{(\frac{1}{2},2)}(|\xi|)e^{-c_{1}(t-s)|2^{j}\xi|^{\gamma_{2}}},\\
\left|\nabla\left(\nabla(-\Delta)^{\lfloor\frac{\eta}{2}\rfloor}h\right)(\xi)\right|
&\leq C_{3}2^{j\gamma_{1}}|\xi|^{\gamma_{1}-2\lfloor\frac{\eta}{2}\rfloor-2}
1_{(\frac{1}{2},2)}(|\xi|)e^{-c_{1}(t-s)|2^{j}\xi|^{\gamma_{2}}}
\end{align*}
for some constants $C_{1},C_{2},C_{3}, c_{1}>0$.
Also, by Lemma \ref{lem: esti 1} with $|\alpha|=0$, we obtain that
\begin{align*}
\|h\|_{L^{2}(\mathbb{R}^{d})}^2
&=\int_{\mathbb{R}^{d}}|h(\xi)|^{2}d\xi\\
&\leq \int_{\mathbb{R}^{d}}C^{2}2^{2j\gamma_{1}}|\xi|^{2\gamma_{1}}1_{(\frac{1}{2},2)}(|\xi|)
         e^{-2c(t-s)|2^{j}\xi|^{\gamma_{2}}}d\xi\\
&\leq C^{2}2^{2j\gamma_{1}}e^{-\frac{2c}{2^{\gamma_{2}}}(t-s)2^{j\gamma_{2}}}
         \int_{\mathbb{R}^{d}}|\xi|^{2\gamma_{1}}1_{(\frac{1}{2},2)}(|\xi|)d\xi\\
&=C_{4}2^{2j\gamma_{1}}e^{-\frac{2c}{2^{\gamma_{2}}}(t-s)2^{j\gamma_{2}}}
\end{align*}
for some constant $C_{4}>0$. Since $\eta=\frac{d+\epsilon}{2}>\frac{d}{2}$, by applying (b) in Corollary \ref{cor: integrability bounded interval} with $f(u)=2^{j\gamma_{1}}e^{-c_{1}(t-s)2^{j\gamma_{2}}u^{\gamma_{2}}}$,
there exists a constant $C_{5}>0$ such that
\begin{align*}
\|L_{\psi_{1}}(l)p_{\psi_{2}}(t,s,\cdot)\|_{L^{1}(\mathbb{R}^{d})}
&=\left\|2^{jd}\mathcal{F}^{-1}\left(h(\xi)\right)(2^{j}\cdot)\right\|_{L^{1}(\mathbb{R}^{d})}
=\left\|\mathcal{F}^{-1}\left(h(\xi)\right)\right\|_{L^{1}(\mathbb{R}^{d})}\\
&\leq C_{5}f\left(\frac{1}{2}\right)
=C_{5}2^{j\gamma_{1}}e^{-\frac{c_{1}}{2^{\gamma_{2}}}(t-s)2^{j\gamma_{2}}}.
\end{align*}
By taking $C=C_{5}$ and $c=\frac{c_{1}}{2^{\gamma_{2}}}$, the proof is complete.
\end{proof}
By considering the Littlewood-Paley decomposition of $f$ given in \eqref{eqn: LP decomposition},
for each $p\geq 1$ and $\alpha\in\mathbb{R}$, let $B_{pp}^{\alpha}(\mathbb{R}^{d})$ be the Besov space with the norm:
\[
\|f\|_{B_{pp}^{\alpha}(\mathbb{R}^{d})}=\|S_{0}(f)\|_{L^{p}(\mathbb{R}^{d})}+\left(\sum_{j=1}^{\infty}2^{j\alpha p}\|\Delta_{j}f\|_{L^{p}(\mathbb{R}^{d})}^{p}\right)^{\frac{1}{p}}.
\]
Also, we denote by $H_{p}^{\alpha}(\mathbb{R}^{d})$ the Sobolev space with the norm:
\[
\|f\|_{H_{p}^{\alpha}(\mathbb{R}^{d})}=\|(1-\Delta)^{\frac{\alpha}{2}}f\|_{L^{p}(\mathbb{R}^{d})}.
\]
It is well known that if $p\geq 2$ and $\alpha\in\mathbb{R}$, then the Sobolev space $H_{p}^{\alpha}(\mathbb{R}^{d})$ is continuously embedded into the Besov space $B_{pp}^{\alpha}(\mathbb{R}^{d})$, i.e., it holds that
\begin{equation}\label{eqn: Hps embedded Bpps}
\|f\|_{B_{pp}^{\alpha}(\mathbb{R}^{d})}\leq C\|f\|_{H_{p}^{\alpha}(\mathbb{R}^{d})}
\end{equation}
for some constant $C>0$ (see, e.g., Theorem 6.4.4 in \cite{Bergh 1976}).

\begin{theorem}\label{lem: LP ineq for p equal lambda fixed s}
Let $(\psi_{1},\psi_{2})\in\mathfrak{S}^{2}$ and let $0< a< \infty$.
Then for any $q\geq 1$ and $f\in B_{qq}^{0}(\mathbb{R}^{d})$,
there exists a constant $C$ depending on $a,q,\boldsymbol{\gamma}, \boldsymbol{\mu}, \boldsymbol{\kappa}$ and $a$ such that
for any $s,l\geq 0$,
\begin{align}
\int_{\mathbb{R}^{d}}\int_{s}^{s+a}(t-s)^{\frac{q\gamma_{1}}{\gamma_{2}}-1}
          |\left(L_{\psi_{1}}(l)p_{\psi_{2}}(t,s,\cdot)\right)*f(x)|^{q}dtdx
 &\leq C\|f\|_{B_{qq}^{0}(\mathbb{R}^{d})}^{q}.
   \label{eqn:FE for LPI}
\end{align}
That is, the convolution operator $Tf=\left(L_{\psi_{1}}(l)p_{\psi_{2}}(\cdot,s,\cdot)\right)*f$ is a bounded linear operator from $B_{qq}^{0}(\mathbb{R}^{d})$ into $L^{q}(\mathbb{R}^{d};V)$ with $V=L^{q}((s,s+a), (t-s)^{\frac{q\gamma_{1}}{\gamma_{2}}-1}dt)$.
\end{theorem}

\begin{proof}
Let $l,s\ge0$, $q\geq 1$ and $f\in B_{qq}^{0}(\mathbb{R}^{d})$ be given.
For a notational convenience, we put
\begin{align*}
\mathcal{C}(t)(\cdot)=\left(L_{\psi_{1}}(l)p_{\psi_{2}}(t,s,\cdot)\right),
\quad
\mathcal{C}_{j}(t)(\cdot)=\left(L_{\psi_{1}}(l)p_{\psi_{2,j}}(t,s,\cdot)\right),
\quad
\left\|\cdot\right\|_{L^{q}}=\left\|\cdot\right\|_{L^{q}(\mathbb{R}^{d})}.
\end{align*}
Then by applying \eqref{eqn: pseudo orthogonal property}, Lemma \ref{lem: boundedness of L psi of p j}
and Young's convolution inequality, for any $q\ge1$, we obtain that
\begin{align*}
\int_{\mathbb{R}^{d}}|\mathcal{C}(t)*f(x)|^{q}dx
&=\int_{\mathbb{R}^{d}}
 \left|\sum_{j=-\infty}^{1}\mathcal{C}_{j}(t)*S_{0}(f)(x)
      +\sum_{j=1}^{\infty}\sum_{j-1\leq i\leq j+1}\mathcal{C}_{i}(t)*f_{j}(x)\right|^{q}dx\\
&\le \left(\sum_{j=-\infty}^{1}\left\|\mathcal{C}_{j}(t)*S_{0}(f)\right\|_{L^{q}}
   +\sum_{j=1}^{\infty}  \sum_{j-1\leq i\leq j+1}  \left\|\mathcal{C}_{i}(t)*f_{j}\right\|_{L^{q}}\right)^{q}\\
&\le 2^{q-1}\left(\sum_{j=-\infty}^{1}\left\|\mathcal{C}_{j}(t)*S_{0}(f)\right\|_{L^{q}}\right)^{q}\\
&\quad\quad
     +2^{q-1}\left(\sum_{j=1}^{\infty}\sum_{j-1\leq i\leq j+1}\left\|\mathcal{C}_{i}(t)*f_{j}\right\|_{L^{q}}\right)^{q}\\
&\le 2^{q-1}\left(\sum_{j=-\infty}^{1}\left\|\mathcal{C}_{j}(t)\right\|_{L^{1}}
     \left\|S_{0}(f)\right\|_{L^{q}}\right)^{q}\\
&\quad\quad +2^{q-1}\left(\sum_{j=1}^{\infty}\sum_{j-1\leq i\leq j+1}
         \left\|\mathcal{C}_{i}(t)\right\|_{L^{1}}
          \left\|f_{j}\right\|_{L^{q}}\right)^{q}\\
&\le 2^{q-1}C^q\left\|S_{0}(f)\right\|_{L^{q}}^q
   \left(\sum_{j=-\infty}^{1}2^{j\gamma_{1}}e^{-c (t-s)2^{j\gamma_{2}}}\right)^{q}\\
&\quad\quad
   +2^{q-1}C^q\left(\sum_{j=1}^{\infty}\sum_{j-1\leq i\leq j+1}
         2^{i\gamma_{1}}e^{-c(t-s)2^{i\gamma_{2}}}
          \left\|f_{j}\right\|_{L^{q}}\right)^{q}.
\end{align*}
Therefore, we have
\begin{align}\label{EI-s-b1}
\int_{s}^{s+a}\int_{\mathbb{R}^{d}}(t-s)^{\frac{q\gamma_{1}}{\gamma_{2}}-1} |\mathcal{C}(t)*f(x)|^{q}dxdt
\le 2^{q-1}C^q  \left(\left\|S_{0}(f)\right\|_{L^{q}}^q A_1+A_2\right),
\end{align}
where
\begin{align*}
A_1&=\int_{s}^{s+a}(t-s)^{\frac{q\gamma_{1}}{\gamma_{2}}-1}
   \left(\sum_{j=-\infty}^{1}2^{j\gamma_{1}}e^{-c(t-s)2^{j\gamma_{2}}}\right)^{q}dt,\\
A_2&=\int_{s}^{s+a}(t-s)^{\frac{q\gamma_{1}}{\gamma_{2}}-1} \left(\sum_{j=1}^{\infty}\sum_{j-1\leq i\leq j+1}
         2^{i\gamma_{1}}e^{-c(t-s)2^{i\gamma_{2}}}
          \left\|f_{j}\right\|_{L^{q}}\right)^{q}dt.
\end{align*}
On the other hand, since
\[
\sum_{j=-\infty}^{1}2^{j\gamma_{1}}e^{-c(t-s)2^{j\gamma_{2}}}
\leq \sum_{j=-\infty}^{1}2^{j\gamma_{1}}
=2^{\gamma_{1}}+\frac{2^{\gamma_{1}}}{2^{\gamma_{1}}-1}=:C_{1},
\]
we have
\begin{align}\label{Estimation of A1}
A_{1}\le C_1^q \int_{s}^{s+a}(t-s)^{\frac{q\gamma_{1}}{\gamma_{2}}-1}dt\leq \frac{C_1^q\gamma_{2}}{q\gamma_{1}}a^{\frac{q\gamma_{1}}{\gamma_{2}}}.
\end{align}
Now, we estimate $A_{2}$. For a notational convenience,
we put $a_{i}=2^{i\gamma_{1}}e^{-ct2^{i\gamma_{2}}}$ for $i\geq 0$.
Then for each $j\geq 1$ it holds that
\[
a_{i}\leq 2^{(j+1)\gamma_{1}}e^{-c(t-s)2^{(j-1)\gamma_{2}}},\quad i=j-1,\,\,j,\,\,j+1.
\]
Therefore, we obtain that
\begin{align*}
\sum_{j=1}^{\infty}\sum_{j-1\leq i\leq j+1}2^{i\gamma_{1}}e^{-c(t-s)2^{i\gamma_{2}}}\left\|f_{j}\right\|_{L^{q}}
&=\sum_{j=1}^{\infty}(a_{j-1}+a_{j}+a_{j+1})\|f_{j}\|_{L^{q}}\\
&\leq 3\sum_{j=1}^{\infty}2^{(j+1)\gamma_{1}}e^{-c(t-s)2^{(j-1)\gamma_{2}}}\|f_{j}\|_{L^{q}}\\
&=3\cdot 2^{2\gamma_{1}}\sum_{j=0}^{\infty}2^{j\gamma_{1}}e^{-c(t-s)2^{j\gamma_{2}}}
    \|f_{j+1}\|_{L^{q}},
\end{align*}
from which by considering $J=\{j\in\mathbb{N}\cup\{0\}\,|\,(t-s)2^{j\gamma_{2}}\leq 1\}$, we see that
\begin{align*}
A_2&=\int_{s}^{s+a} (t-s)^{\frac{q\gamma_{1}}{\gamma_{2}}-1}\left(\sum_{j=1}^{\infty}\sum_{j-1\leq i\leq j+1}
         2^{i\gamma_{1}}e^{-c(t-s)2^{i\gamma_{2}}}
          \left\|f_{j}\right\|_{L^{q}}\right)^{q}dt\\
&\le 3^q\cdot 2^{2q\gamma_{1}}\int_{s}^{s+a}(t-s)^{\frac{q\gamma_{1}}{\gamma_{2}}-1}
     \left(\sum_{j=0}^{\infty}2^{j\gamma_{1}}e^{-c(t-s)2^{j\gamma_{2}}}\|f_{j+1}\|_{L^{q}}\right)^qdt\\
&\le 3^{q}\cdot 2^{2q\gamma_{1}}2^{q-1}(B_1+B_2),
\end{align*}
where
\begin{align*}
B_1:&=\int_{s}^{s+a}(t-s)^{\frac{q\gamma_{1}}{\gamma_{2}}-1}\left(\sum_{j\in J}2^{j\gamma_{1}}e^{-c(t-s)2^{j\gamma_{2}}}
         \|f_{j+1}\|_{L^{q}}\right)^{q}dt,\\
B_2:&=\int_{s}^{s+a}(t-s)^{\frac{q\gamma_{1}}{\gamma_{2}}-1}\left(\sum_{j\notin J}2^{j\gamma_{1}}e^{-c(t-s)2^{j\gamma_{2}}}
         \|f_{j+1}\|_{L^{q}}\right)^{q}dt.
\end{align*}
For $B_{1}$, we note that for all $j\in J$ with $2^{j\gamma_{2}}\leq (t-s)^{-1}$, it holds that
\[
\sum_{j\in J}2^{\frac{j\gamma_{1} q }{2(q-1)}}
=\sum_{j\in J}\left(2^{j\gamma_{2}}\right)^{\frac{\gamma_{1} q }{2(q-1)\gamma_{2}}}
\leq \sum_{j\in J}(t-s)^{-\frac{ q\gamma_{1}}{2(q-1)\gamma_{2}}}
= |J|(t-s)^{-\frac{ q\gamma_{1}}{2(q-1)\gamma_{2}}},
\]
where $|J|$ is the cardinal number of $J$.
By H\"{o}lder inequality and the fact that
\begin{align*}
e^{-c(t-s)2^{j\gamma_{2}}}\leq 1,
\end{align*}
and by using
 \begin{align*}
 1_{\{(t-s)2^{j\gamma_{2}}\leq 1\}}(j)
 &=\left\{
     \begin{array}{ll}
       1, & \hbox{$(t-s)2^{j\gamma_{2}}\leq 1$,} \\
       0, & \hbox{otherwise.}
     \end{array}
   \right.
=\left\{
     \begin{array}{ll}
       1, & \hbox{$t\leq s+2^{-j\gamma_{2}}$,} \\
       0, & \hbox{otherwise.}
     \end{array}
   \right.\\
 &=1_{[s,s+2^{-j\gamma_{2}}]}(t),
 \end{align*}
we obtain that
\begin{align*}
B_1&=\int_{s}^{s+a}(t-s)^{\frac{q\gamma_{1}}{\gamma_{2}}-1}\left(\sum_{j\in J}2^{\frac{j\gamma_{1}}{2}}2^{\frac{j\gamma_{1}}{2}}e^{-ct2^{j\gamma_{2}}}
         \|f_{j+1}\|_{L^{q}}\right)^{q}dt\\
&\le \int_{s}^{s+a}(t-s)^{\frac{q\gamma_{1}}{\gamma_{2}}-1}\left(\sum_{j\in J}2^{\frac{j\gamma_{1} q}{2(q-1)}}\right)^{q-1}
   \sum_{j\in J}2^{\frac{j\gamma_{1} q}{2}}e^{-cq t2^{j\gamma_{2}}}\|f_{j+1}\|_{L^{q}}^{q}dt\\
&\le |J|^{q-1}\int_{s}^{s+a}(t-s)^{\frac{q\gamma_{1}}{2\gamma_{2}}-1}
   \sum_{j=0}^{\infty}1_{\{(t-s)2^{j\gamma_{2}}\leq 1\}}(j)2^{\frac{j\gamma_{1} q}{2}}
   \|f_{j+1}\|_{L^{q}}^{q}dt\\
&=|J|^{q-1}\sum_{j=0}^{\infty}
  \int_{s}^{s+a}(t-s)^{\frac{q\gamma_{1}}{2\gamma_{2}}-1}1_{[s,s+2^{-j\gamma_{2}}]}(t)2^{\frac{j\gamma_{1} q}{2}}
  \|f_{j+1}\|_{L^{q}}^{q}dt\\
&=|J|^{q-1}\sum_{j=0}^{\infty}2^{\frac{j\gamma_{1} q}{2}}\|f_{j+1}\|_{L^{q}}^{q}
       \int_{s}^{s+(a \wedge 2^{-j\gamma_{2}})}(t-s)^{\frac{q\gamma_{1}}{2\gamma_{2}}-1}dt\\
&\le |J|^{q-1}\sum_{j=0}^{\infty}2^{\frac{j\gamma_{1} q}{2}}\|f_{j+1}\|_{L^{q}}^{q}\frac{2\gamma_{2}}{q\gamma_{1}}2^{-\frac{j\gamma_{1} q}{2}}\\
&=\frac{2\gamma_{2}|J|^{q-1}}{q\gamma_{1}}\sum_{j=0}^{\infty}\|f_{j+1}\|_{L^{q}}^{q}.
\end{align*}
In order to estimate $B_{2}$, we note that
\[
\sum_{j\notin J}e^{-c(t-s)2^{j\gamma_{2}}}\leq C_{2}
\]
for some constant $C_{2}>0$.
By applying the H\"{o}lder's inequality with the conjugate number $q'$ of $q$
and changing variable,
for some constant $C_{3}>0$, we obtain that
\begin{align*}
B_2&=\int_{s}^{s+a}(t-s)^{\frac{q\gamma_{1}}{\gamma_{2}}-1}\left(\sum_{j\notin J}2^{j\gamma_{1} }e^{-c(t-s)2^{j\gamma_{1}}(\frac{1}{q'}+\frac{1}{q})}
         \|f_{j+1}\|_{L^{q}}\right)^{q}dt\\
&\le \int_{s}^{s+a}(t-s)^{\frac{q\gamma_{1}}{\gamma_{2}}-1}\left(\sum_{j\notin J}e^{-c(t-s)2^{j\gamma_{2}}}\right)^{\frac{q}{q'}}
                  \sum_{j\notin J}2^{j\gamma_{1} q}e^{-c(t-s)2^{j\gamma_{2}}}
         \|f_{j+1}\|_{L^{q}}^{q}dt\\
&= C_{3}\sum_{j\notin J}\|f_{j+1}\|_{L^{q}}^{q}
    \int_{0}^{c2^{j\gamma_{2}}a}u^{\frac{q\gamma_{1}}{\gamma_{2}}-1}e^{-u}du\\
&\leq C_{3} \sum_{j\notin J}\|f_{j+1}\|_{L^{q}}^{q}
        \int_{0}^{\infty}u^{\frac{q\gamma_{1}}{\gamma_{2}}-1}e^{-u}du\\
&\leq C_{3}  \Gamma\left(\frac{q\gamma_{1}}{\gamma_{2}}\right)
   \sum_{j=0}^{\infty}\|f_{j+1}\|_{L^{q}}^{q},
\end{align*}
where $\Gamma$ is the gamma function.
Therefore, we obtain that
\begin{align}
A_{2}
&\le 3^{q}\cdot 2^{2q\gamma_{1}}2^{q-1}(B_1+B_2)\nonumber\\
&\le 3^{q}\cdot 2^{2q\gamma_{1}}2^{q-1}\left(\frac{2\gamma_{2}|J|^{q-1}}{q\gamma_{1}}\sum_{j=0}^{\infty}\|f_{j+1}\|_{L^{q}}^{q}
  +C_{3}\Gamma\left(\frac{q\gamma_{1}}{\gamma_{2}}\right)
     \sum_{j=0}^{\infty}\|f_{j+1}\|_{L^{q}}^{q}\right)\nonumber\\
&\le M_1\sum_{j=1}^{\infty} \|f_{j}\|_{L^{q}}^{q},
\quad
     M_1=3^{q}\cdot 2^{2q\gamma_{1}}2^{q-1}
     \max\left\{\frac{2\gamma_{2}|J|^{q-1}}{q\gamma_{1}},C_{3}\Gamma\left(\frac{q\gamma_{1}}{\gamma_{2}}\right)\right\}.
     \label{Estimation of A2}
\end{align}
Finally, by combining \eqref{EI-s-b1}, \eqref{Estimation of A1} and \eqref{Estimation of A2},
we obtain that
\begin{align*}
\int_{s}^{s+a}\int_{\mathbb{R}^{d}}(t-s)^{\frac{q\gamma_{1}}{\gamma_{2}}-1}
|\mathcal{C}(t)*f(x)|^{q}dxdt
&\le 2^{q-1}C^q\left(\left\|S_{0}(f)\right\|_{L^{q}}^q A_1+A_2\right)\\
&\le 2^{q-1}C^q\left(A_{1}  \left\|S_{0}(f)\right\|_{L^{q}}^q
  +M_1\sum_{j=1}^{\infty} \|f_{j}\|_{L^{q}}^{q}\right)\\
&\le M_{2}\left(\left\|S_{0}(f)\right\|_{L^{q}}^q
                        +\sum_{j=1}^{\infty} \|f_{j}\|_{L^{q}}^{q}\right)\\
&\le M_{2}\|f\|_{B_{qq}^{0}(\mathbb{R}^{d})}^{q}
\end{align*}
with
\begin{align*}
M_2=2^{q-1}C^q
  \max\left\{A_{1},M_{1}\right\},
\end{align*}
which implies \eqref{eqn:FE for LPI}.
\end{proof}

\section{Boundedness of Convolution Operators}
\label{sec:Boundedness of Convolution Operators}

In this section, we prove the boundedness of the convolution operators
corresponding to the compositions of pseudo-differential operators and evolution systems
associated with the symbols of pseudo-differential operators.

\begin{theorem} \label{thm: LP ineq p equal lambda}
Let $(\psi_{1},\psi_{2})\in\mathfrak{S}^{2}$.
Let $0< a< \infty$ and let $q\geq 2$. Then it holds that
\begin{itemize}
  \item [\rm{(i)}] there exists a constant $C_{1}>0$ depending on $a,d,q,\boldsymbol{\gamma},\boldsymbol{\mu},$ and $\boldsymbol{\kappa}$ such that for any $s,l\geq 0$ and $f\in L^{q}(\mathbb{R}^{d})$,
\begin{equation}\label{eqn: LP 1}
\int_{s}^{s+a}\int_{\mathbb{R}^{d}}(t-s)^{\frac{q\gamma_{1}}{\gamma_{2}}-1}
       |\left(L_{\psi_{1}}(l)p_{\psi_{2}}(t,s,\cdot)\right)*f(x)|^{q}dxdt
\leq C_{1}\int_{\mathbb{R}^{d}}|f(x)|^{q}dx,
\end{equation}
i.e., the convolution operator $Tf=\left(L_{\psi_{1}}(l)p_{\psi_{2}}(\cdot,s,\cdot)\right)*f$ is a bounded linear operator from $L^{q}(\mathbb{R}^{d})$ into $L^{q}(\mathbb{R}^{d};V)$ with $V=L^{q}((s,s+a), (t-s)^{\frac{q\gamma_{1}}{\gamma_{2}}-1}dt)$,
  \item [\rm{(ii)}] if $\psi_{1}(l,\xi)=\psi_{1}(\xi)$, $\psi_{2}(r,\xi)=\psi_{2}(\xi)$,
i.e., $\psi_{1}(l,\xi)$ and $\psi_{2}(r,\xi)$ are constant with respect to the variables $l$ and $r$, respectively,
and $\psi_{1},\psi_{2}$ satisfy the homogeneity:
\begin{align}\label{eqn:homogenioty of psi}
\psi_{1}(\lambda\xi)=\lambda^{\gamma_{1}}\psi_{1}(\xi),\quad
\psi_{2}(\lambda\xi)=\lambda^{\gamma_{2}}\psi_{2}(\xi),\quad\text{for}\quad\lambda>0,
\end{align}
then there exists a constant $C_2>0$ depending on $d,q,\boldsymbol{\gamma},\boldsymbol{\mu},$ and $\boldsymbol{\kappa}$
such that for any $s,l\ge0$ and $f\in L^{q}(\mathbb{R}^{d})$,
\begin{equation}\label{eqn: LP 3}
\int_{s}^{\infty}\int_{\mathbb{R}^{d}}(t-s)^{\frac{q\gamma_{1}}{\gamma_{2}}-1}
         |\left(L_{\psi_{1}}p_{\psi_{2}}(t,s,\cdot)\right)*f(x)|^{q}dxdt
\leq C_{2}\int_{\mathbb{R}^{d}}|f(x)|^{q}dx,
\end{equation}
i.e., the convolution operator $Tf=\left(L_{\psi_{1}}p_{\psi_{2}}(\cdot,s,\cdot)\right)*f$ is a bounded linear operator from $L^{q}(\mathbb{R}^{d})$ into $L^{q}(\mathbb{R}^{d};V)$ with $V=L^{q}((s,\infty), (t-s)^{\frac{q\gamma_{1}}{\gamma_{2}}-1}dt)$,
  \item [\rm{(iii)}] if $q=2$, then there exists a constant $C_{3}>0$ depending on $\mu_{1},\kappa_{2}$, and $\boldsymbol{\gamma}$
such that for any $s,l\ge0$ and $f\in L^{2}(\mathbb{R}^{d})$,
\begin{equation}\label{eqn: LP 2}
\int_{s}^{\infty}\int_{\mathbb{R}^{d}}(t-s)^{\frac{2\gamma_{1}}{\gamma_{2}}-1}
      |\left(L_{\psi_{1}}(l)p_{\psi_{2}}(t,s,\cdot)\right)*f(x)|^{2}dxdt
\leq C_{3}\int_{\mathbb{R}^{d}}|f(x)|^{2}dx,
\end{equation}
i.e., the convolution operator $Tf=\left(L_{\psi_{1}}(l)p_{\psi_{2}}(\cdot,s,\cdot)\right)*f$ is a bounded linear operator from $L^{2}(\mathbb{R}^{d})$ into $L^{2}(\mathbb{R}^{d};V)$ with $V=L^{2}((s,\infty), (t-s)^{\frac{2\gamma_{1}}{\gamma_{2}}-1}dt)$.
\end{itemize}
\end{theorem}

\begin{proof}
(i) \enspace By applying Theorem \ref{lem: LP ineq for p equal lambda fixed s} and \eqref{eqn: Hps embedded Bpps} with $\alpha=0$,
we see that there exist constants $C,\widetilde{C}>0$ such that
\begin{align}\label{ineq: int s b int Rd is bounded by L lambda norm}
\int_{s}^{s+a}\int_{\mathbb{R}^{d}}(t-s)^{\frac{q\gamma_{1}}{\gamma_{2}}-1}
    |\left(L_{\psi_{1}}(l)p_{\psi_{2}}(t,s,\cdot)\right)*f(x)|^{q}dxdt
&\leq C\|f\|_{B_{qq}^{0}(\mathbb{R}^{d})}^{q}\nonumber\\
&\leq C\widetilde{C}\|f\|_{H_{q}^{0}(\mathbb{R}^{d})}^{q}\nonumber\\
&=C\widetilde{C}\|f\|_{L^{q}(\mathbb{R}^{d})}^{q},
\end{align}
and then by taking $C_{1}=C\widetilde{C}$, we get \eqref{eqn: LP 1}.

(ii) \enspace
For notational convenience, for each $b>0$, we put $f_{b}(x):=f(b^{\frac{1}{\gamma_{2}}}x)$.
Then we have
\begin{align*}
(\mathcal{F}f_{b})(\xi)=b^{-\frac{d}{\gamma_{2}}}(\mathcal{F}f)(b^{-\frac{1}{\gamma_{2}}}\xi),
\end{align*}
which, from the homogeneity of $\psi_{1}$ and $\psi_{2}$, implies that
\begin{align}
L_{\psi_{1}}p_{\psi_{2}}(bt+s,s,\cdot)*f(x)
&=\mathcal{F}^{-1}(\psi_{1}(\xi)e^{bt\psi_{2}(\xi)}\mathcal{F}f(\xi))(x)\nonumber\\
&=\mathcal{F}^{-1}(\psi_{1}(\xi)e^{t\psi_{2}(b^{\frac{1}{\gamma_{2}}}\xi)}\mathcal{F}f(\xi))(x)\nonumber\\
&=b^{-\frac{\gamma_{1}}{\gamma_{2}}-\frac{d}{\gamma_{2}}}\mathcal{F}^{-1}
  (\psi_{1}(\xi)e^{t\psi_{2}(\xi)}\mathcal{F}f(b^{-\frac{1}{\gamma_{2}}}\xi))(b^{-\frac{1}{\gamma_{2}}}x)\nonumber\\
&=b^{-\frac{\gamma_{1}}{\gamma_{2}}}\left(\left(L_{\psi_{1}}p_{\psi_{2}}(t,0,\cdot)\right)*f_{b}\right)(b^{-\frac{1}{\gamma_{2}}}x).
               \label{eqn: LT bt+s}
\end{align}
Then by applying \eqref{eqn: LT bt+s}, \eqref{ineq: int s b int Rd is bounded by L lambda norm} with $s=0$ and $a=1$,
and the change of variables,
we obtain that
\begin{align}\label{eqn: ineq with constant independent of a}
&\int_{s}^{s+b}\int_{\mathbb{R}^{d}}(t-s)^{\frac{q\gamma_{1}}{\gamma_{2}}-1}
    |\left(L_{\psi_{1}}p_{\psi_{2}}(t,s,\cdot)\right)*f(x)|^{q}dxdt\nonumber\\
&=b^{\frac{q\gamma_{1}}{\gamma_{2}}}\int_{0}^{1}\int_{\mathbb{R}^{d}}t^{\frac{q\gamma_{1}}{\gamma_{2}}-1}
    |\left(L_{\psi_{1}}p_{\psi_{2}}(bt+s,s,\cdot)\right)*f(x)|^{q}dxdt\nonumber\\
&=\int_{0}^{1}\int_{\mathbb{R}^{d}}t^{\frac{q\gamma_{1}}{\gamma_{2}}-1}
    |\left(\left(L_{\psi_{1}}p_{\psi_{2}}(t,0,\cdot)\right)*f_{b}\right)(b^{-\frac{1}{\gamma_{2}}}x)|^{q}dxdt\nonumber\\
&=b^{\frac{d}{\gamma_{2}}}\int_{0}^{1}\int_{\mathbb{R}^{d}}t^{\frac{q\gamma_{1}}{\gamma_{2}}-1}
    |\left(\left(L_{\psi_{1}}p_{\psi_{2}}(t,0,\cdot)\right)*f_{b}\right)(x)|^{q}dxdt\nonumber\\
&\leq b^{\frac{d}{\gamma_{2}}}C_{q,d,\boldsymbol{\gamma},\boldsymbol{\kappa},\boldsymbol{\mu}}
    \int_{\mathbb{R}^{d}}|f_{b}(x)|^{q}dx\nonumber\\
&=b^{\frac{d}{\gamma_{2}}}C_{q,d,\boldsymbol{\gamma},\boldsymbol{\kappa},\boldsymbol{\mu}}\int_{\mathbb{R}^{d}}
    |f(b^{\frac{1}{\gamma_{2}}}x)|^{q}dx\nonumber\\
&=C_{q,d,\boldsymbol{\gamma},\boldsymbol{\kappa},\boldsymbol{\mu}}\int_{\mathbb{R}^{d}}|f(x)|^{q}dx.
\end{align}
Since $b>0$ is arbitrary, by taking $b\rightarrow\infty$ in \eqref{eqn: ineq with constant independent of a},
we have the inequality given in \eqref{eqn: LP 3}.

(iii) \enspace Let $q=2$. Then by Plancherel theorem and Assumptions \textbf{(S1)} and \textbf{(S2)},
we obtain that
\begin{align*}
&\int_{s}^{\infty}\int_{\mathbb{R}^{d}}(t-s)^{\frac{2\gamma_{1}}{\gamma_{2}}-1}
   |\left(L_{\psi_{1}}(l)p_{\psi_{2}}(t,s,\cdot)\right)*f(x)|^{2}dxdt\\
&=\int_{s}^{\infty}\int_{\mathbb{R}^{d}}(t-s)^{\frac{2\gamma_{1}}{\gamma_{2}}-1}\left|\psi_{1}(l,\xi)\exp\left(\int_{s}^{t}
\psi_{2}(r,\xi)dr\right)\mathcal{F}f(\xi)\right|^{2}d\xi dt\\
&\leq \int_{s}^{\infty}\int_{\mathbb{R}^{d}}(t-s)^{\frac{2\gamma_{1}}{\gamma_{2}}-1}
   \mu_{1}^{2}|\xi|^{2\gamma_{1}}e^{-2\kappa_{2}(t-s)|\xi|^{\gamma_{2}}}\left|\mathcal{F}f(\xi)\right|^{2}d\xi dt\\
&=\mu_{1}^{2}\int_{\mathbb{R}^{d}}|\xi|^{2\gamma_{1}}
  \left(\int_{s}^{\infty}(t-s)^{\frac{2\gamma_{1}}{\gamma_{2}}-1}e^{-2\kappa_{2}(t-s)|\xi|^{\gamma_{2}}}dt\right)
  \left|\mathcal{F}f(\xi)\right|^{2} d\xi\\
&=\mu_{1}^{2}\int_{\mathbb{R}^{d}}|\xi|^{2\gamma_{1}}\Gamma
  \left(\frac{2\gamma_{1}}{\gamma_{2}}\right)(2\kappa_{2}|\xi|^{\gamma_{2}})^{-\frac{2\gamma_{1}}{\gamma_{2}}}
    \left|\mathcal{F}f(\xi)\right|^{2} d\xi\\
&=\mu_{1}^{2}\Gamma\left(\frac{2\gamma_{1}}{\gamma_{2}}\right)(2\kappa_{2})^{-\frac{2\gamma_{1}}{\gamma_{2}}}
   \int_{\mathbb{R}^{d}}\left|f(x)\right|^{2} dx,
\end{align*}
which proves \eqref{eqn: LP 2}.
\end{proof}

Let $(V_{1},\|\cdot\|_{V_{1}}), (V_{2},\|\cdot\|_{V_{2}})$ be Banach spaces
and let $\mathcal{B}(V_{1},V_{2})$ be the space of all bounded linear operators from $V_{1}$ into $V_{2}$.

For a Banach space $V$ and $1\leq p< \infty$,
we denote by $L^{p}(\mathbb{R}^{d};V)$ the space of all strongly measurable functions
$f:\mathbb{R}^{d}\rightarrow V$ with the norm
\[
\|f\|_{L^{p}(\mathbb{R}^{d};V)}=\left(\int_{\mathbb{R}^{d}}\|f(x)\|_{V}^{p}dx\right)^{\frac{1}{p}}.
\]

The following theorem is known as the Calder\'{o}n-Zygmund theorem for vector-valued functions.

\begin{theorem}[Theorem 1.5.3 in \cite{Stein 1993}]\label{thm: IKO is bounded in Lp}
Let $V_{1},V_{2}$ be Banach spaces and let $K:\mathbb{R}^{d}\times \mathbb{R}^{d}\rightarrow \mathcal{B}(V_{1},V_{2})$ be a strongly measurable and locally integrable function. For any bounded strongly measurable $V_{1}$-valued function $f$ with compact support defined on $\mathbb{R}^{d}$, the operator
\[
Tf(x):=\int_{\mathbb{R}^{d}}K(x,y)f(y)dy,
\]
is well defined.
Suppose that $T$ is bounded from $L^{q}(\mathbb{R}^{d};V_{1})$ into $L^{q}(\mathbb{R}^{d};V_{2})$ for some $q>1$.
Assume that there exists a constant $A>0$ such that
\begin{equation}\label{eqn: kernel condition for IKO}
\int_{|x-y|\geq 2|y-z|}\|K(x,y)-K(x,z)\|_{\rm OP}dx\leq A
\end{equation}
for all $y, z\in \mathbb{R}^{d}$ with $y\ne z$, where $\|\cdot\|_{\rm OP}$ is the operator norm.
Then for any $1<p<\infty$, there exists a constant $C>0$ depending on $p,d,A$ and $\|T\|_{\rm OP}$ (depending on $q$) such that
\begin{equation}\label{eqn: IKO is bounded}
\|Tf\|_{L^{p}(\mathbb{R}^{d};V_{2})}\leq C\|f\|_{L^{p}(\mathbb{R}^{d};V_{1})}
\end{equation}
for all $f\in L^{p}(\mathbb{R}^{d};V_{1})$.
\end{theorem}

\begin{corollary}\label{cor: boundedness of singular integral}
Let $V$ be a Banach space and
let $K:\mathbb{R}^{d}\rightarrow V$ be a strongly measurable and locally integrable function.
For any bounded measurable real-valued function $f$ with compact support defined on $\mathbb{R}^{d}$,
the operator
\[
Tf(x):=\int_{\mathbb{R}^{d}}K(x-y)f(y)dy.
\]
is well defined.
Suppose that $T$ is bounded from $L^{q}(\mathbb{R}^{d})$ into $L^{q}(\mathbb{R}^{d};V)$ for some $q>1$.
Assume that there exists a constant $A>0$ such that
\begin{equation}\label{eqn: kernel condition for IKO 2}
\int_{|x|\geq 2|y|}\|K(x-y)-K(x)\|_{V}dx\leq A
\end{equation}
for all $y\in \mathbb{R}^{d}\setminus\{0\}$.
Then for any $1<p<\infty$, there exists a constant $C>0$ depending on $p,d,A$ and $\|T\|_{\rm OP}$ (depending on $q$) such that
\begin{equation}\label{eqn: IKO is bounded 2}
\|Tf\|_{L^{p}(\mathbb{R}^{d};V)}\leq C\|f\|_{L^{p}(\mathbb{R}^{d})}
\end{equation}
for all $f\in L^{p}(\mathbb{R}^{d})$.
\end{corollary}

\begin{proof}
This is the case of $K(x,y)=K(x-y)$, $V_{1}=\mathbb{R}$ and $V_{2}=V$ in Theorem \ref{thm: IKO is bounded in Lp}.
\end{proof}

\begin{theorem}\label{thm: convolution is bdd on Lp}
Let $(\psi_{1},\psi_{2})\in\mathfrak{S}^{2}$
such that $N_1,N_2> d+1+\lfloor\gamma_{1}\rfloor$.
Let $s,l \geq 0$ and $0<a<\infty$. Let $Tf=\left(L_{\psi_{1}}(l)p_{\psi_{2}}(\cdot,s,\cdot)\right)*f$.
Then for any $q\geq 2$ and $1<p<\infty$, it holds that
\begin{itemize}
  \item [\rm{(i)}]
$T\in\mathcal{B}(L^{p}(\mathbb{R}^{d}),L^{p}(\mathbb{R}^{d};V))$ with $V=L^{q}((s,s+a), (t-s)^{\frac{q\gamma_{1}}{\gamma_{2}}-1}dt)$,
  \item [\rm{(ii)}]
  if $\psi_{1}(l,\xi)=\psi_{1}(\xi)$, $\psi_{2}(r,\xi)=\psi_{2}(\xi)$,
i.e., $\psi_{1}(l,\xi)$ and $\psi_{2}(r,\xi)$ are constant with respect to the variables $l$ and $r$, respectively,
and $\psi_{1},\psi_{2}$ satisfy the homogeneity given in \eqref{eqn:homogenioty of psi},
then $T\in\mathcal{B}(L^{p}(\mathbb{R}^{d}),L^{p}(\mathbb{R}^{d};V))$ with $V=L^{q}((s,\infty), (t-s)^{\frac{q\gamma_{1}}{\gamma_{2}}-1}dt)$,
  \item [\rm{(iii)}]
  if $q=2$, then $ T\in\mathcal{B}(L^{p}(\mathbb{R}^{d}),L^{p}(\mathbb{R}^{d};V))$
  with $V=L^{2}((s,\infty), (t-s)^{\frac{2\gamma_{1}}{\gamma_{2}}-1}dt)$.
\end{itemize}
\end{theorem}

\begin{proof}
 (i) \enspace By (i) in Theorem \ref{thm: LP ineq p equal lambda}, for any $q\geq 2$, $T\in\mathcal{B}(L^{q}(\mathbb{R}^{d}),L^{q}(\mathbb{R}^{d};V))$.
On the other hand, by Proposition \ref{prop: estimation of int of Kt(x-y)-Kt(x) over outside ball},
the condition given in \eqref{eqn: kernel condition for IKO 2}
with $K(x)=L_{\psi_{1}}(l)p_{\psi_{2}}(\cdot,s,x)$ (for $x\in\mathbb{R}^{d}$)
and $V=L^{q}((s,s+a), (t-s)^{\frac{q\gamma_{1}}{\gamma_{2}}-1}dt)$ holds.
Therefore, by Corollary \ref{cor: boundedness of singular integral},
it holds that
$T\in\mathcal{B}(L^{p}(\mathbb{R}^{d}),L^{p}(\mathbb{R}^{d};V))$ for all $1<p<\infty$.

(ii)\enspace By (ii) in Theorem \ref{thm: LP ineq p equal lambda},
for any $q\geq 2$, $T\in\mathcal{B}(L^{q}(\mathbb{R}^{d}),L^{q}(\mathbb{R}^{d};V))$.
On the other hand, by Proposition \ref{prop: estimation of int of Kt(x-y)-Kt(x) over outside ball},
the condition given in \eqref{eqn: kernel condition for IKO 2}
with $K(x)=L_{\psi_{1}}(l)p_{\psi_{2}}(\cdot,s,x)$ (for $x\in\mathbb{R}^{d}$)
and $V=L^{q}((s,\infty), (t-s)^{\frac{q\gamma_{1}}{\gamma_{2}}-1}dt)$ holds.
Therefore, by Corollary \ref{cor: boundedness of singular integral},
$T\in\mathcal{B}(L^{p}(\mathbb{R}^{d}),L^{p}(\mathbb{R}^{d};V))$ for all $1<p<\infty$.

(iii)\enspace By (iii) in Theorem \ref{thm: LP ineq p equal lambda},
it holds that $T\in\mathcal{B}(L^{2}(\mathbb{R}^{d}),L^{2}(\mathbb{R}^{d};V))$.
On the other hand, by Proposition \ref{prop: estimation of int of Kt(x-y)-Kt(x) over outside ball},
the condition given in \eqref{eqn: kernel condition for IKO 2}
with $K(x)=L_{\psi_{1}}(l)p_{\psi_{2}}(\cdot,s,x)$ (for $x\in\mathbb{R}^{d}$)
and $V=L^{2}((s,\infty), (t-s)^{\frac{2\gamma_{1}}{\gamma_{2}}-1}dt)$ holds.
Therefore, by Corollary \ref{cor: boundedness of singular integral},
it holds that
$T\in\mathcal{B}(L^{p}(\mathbb{R}^{d}),L^{p}(\mathbb{R}^{d};V))$ for all $1<p<\infty$.
\end{proof}

\section{Littlewood-Paley Type Inequality}\label{sec:Main Results}

The following is a main theorem in this paper,
which is a Littlewood-Paley type inequality for a pseudo-differential operator
and an evolution system associated with a symbol of a pseudo-differential operator.

\begin{theorem}\label{thm: Lp boundedness of S lambda infty}
Let $(\psi_{1},\psi_{2})\in\mathfrak{S}^{2}$.
Let $0<a<\infty$ be given.
Suppose that $N_1,N_2> d+1+\lfloor\gamma_{1}\rfloor$.
Then for any $q\geq 2$ and $1<p<\infty$, it holds that
\begin{itemize}
  \item [\rm{(i)}] there exists a constant $C_{1}>0$
depending on $a, d, p, q,\boldsymbol{\gamma},\boldsymbol{\mu}$ and $\boldsymbol{\kappa}$ such that
for any $s,l\geq 0$ and $f\in L^{p}(\mathbb{R}^{d})$,
\begin{equation}\label{eqn: LP ineq for q}
\int_{\mathbb{R}^{d}}\left(\int_{s}^{s+a}(t-s)^{\frac{q\gamma_{1}}{\gamma_{2}}-1}
   |L_{\psi_{1}}(l)\mathcal{T}_{\psi_{2}}(t,s)f(x)|^{q}dt\right)^{\frac{p}{q}}dx
\leq C_{1}\int_{\mathbb{R}^{d}}|f(x)|^{p}dx,
\end{equation}
  \item [\rm{(ii)}]
  if $\psi_{1}(l,\xi)=\psi_{1}(\xi)$, $\psi_{2}(r,\xi)=\psi_{2}(\xi)$,
i.e., $\psi_{1}(l,\xi)$ and $\psi_{2}(r,\xi)$ are constant with respect to the variables $l$ and $r$, respectively,
and $\psi_{1},\psi_{2}$ satisfy the homogeneity given in \eqref{eqn:homogenioty of psi},
then there exists a constant $C_{2}>0$ depending on $d,p,q,\boldsymbol{\gamma},\boldsymbol{\mu},$ and $\boldsymbol{\kappa}$
such that for any $s\geq 0$ and $f\in L^{p}(\mathbb{R}^{d})$,
\begin{equation}\label{eqn: LP ineq for homogeneous}
\int_{\mathbb{R}^{d}}\left(\int_{s}^{\infty}(t-s)^{\frac{q\gamma_{1}}{\gamma_{2}}-1}
 |L_{\psi_{1}}\mathcal{T}_{\psi_{2}}(t,s)f(x)|^{q}dt\right)^{\frac{p}{q}}dx
\leq C_{2}\int_{\mathbb{R}^{d}}|f(x)|^{p}dx,
\end{equation}
  \item [\rm{(iii)}]
  if $q=2$, then there exists a constant $C_{3}>0$
  depending on $d,p,\boldsymbol{\gamma},\boldsymbol{\mu},$ and $\boldsymbol{\kappa}$ such that for any $s,l\geq 0$
  and $f\in L^{p}(\mathbb{R}^{d})$,
\begin{equation}\label{eqn: LP ineq q=2}
\int_{\mathbb{R}^{d}}\left(\int_{s}^{\infty}(t-s)^{\frac{2\gamma_{1}}{\gamma_{2}}-1}
    |L_{\psi_{1}}(l)\mathcal{T}_{\psi_{2}}(t,s)f(x)|^{2}dt\right)^{\frac{p}{2}}dx
\leq C_{3}\int_{\mathbb{R}^{d}}|f(x)|^{p}dx.
\end{equation}
\end{itemize}
\end{theorem}

\begin{proof}
From \eqref{eqn: kernel}, for any $f\in L^{p}(\mathbb{R}^{d})$, it is clear that
\begin{align*}
L_{\psi_{1}}(l)\mathcal{T}_{\psi_{2}}(t,s)f(x)
&=\left(L_{\psi_{1}}(l)p_{\psi_{2}}(t,s,\cdot)\right)*f(x),
\end{align*}
and then by applying Theorem \ref{thm: convolution is bdd on Lp}, the proof is immediate.
\end{proof}

\begin{corollary}
Let $q\geq 2$ and let $0<a<\infty$.
Let $\psi\in\mathfrak{S}$ such that $N>d+1+\lfloor\frac{\gamma}{q}\rfloor$, where $N$ is in \textbf{(S2)}.
Then for any $1<p<\infty$, it holds that
\begin{itemize}
   \item [\rm{(i)}] there exists a constant $C_{1}>0$ depending on $d,a,q,p,\kappa,\mu$ and $\gamma$ such that
   for any $s\geq 0$ and $f\in L^{p}(\mathbb{R}^{d})$,
\begin{equation}\label{eqn: LP ineq Laplacian}
\int_{\mathbb{R}^{d}}\left(\int_{s}^{s+a}|(-\Delta)^{\frac{\gamma}{2q}}\mathcal{T}_{\psi}(t,s)f(x)|^{q}dt\right)^{\frac{p}{q}}dx
\leq C_{1}\int_{\mathbb{R}^{d}}|f(x)|^{p}dx,
\end{equation}
   \item [\rm{(ii)}] if $\psi(r,\xi)=\psi(\xi)$,
i.e., $\psi(r,\xi)$ is constant with respect to the variable $r$,
and $\psi$ satisfies the homogeneity given in \eqref{eqn:homogenioty of psi},
then there exists a constant $C_{2}>0$ depending on $d,q,p,\kappa,\mu,$ and $\gamma$
such that for any $s\geq 0$ and $f\in L^{p}(\mathbb{R}^{d})$,
\begin{equation}\label{eqn: LP ineq Laplacian homo}
\int_{\mathbb{R}^{d}}\left(\int_{s}^{\infty}|(-\Delta)^{\frac{\gamma}{2q}}\mathcal{T}_{\psi}(t,s)f(x)|^{q}dt\right)^{\frac{p}{q}}dx
\leq C_{2}\int_{\mathbb{R}^{d}}|f(x)|^{p}dx,
\end{equation}
   \item [\rm{(iii)}]
   if $q=2$, there exists a constant $C_{3}>0$ depending on $d,p,\kappa,\mu$ and $\gamma$ such that
for any $s\geq 0$ and $f\in L^{p}(\mathbb{R}^{d})$,
\begin{equation}\label{eqn: LP ineq Laplacian q=2}
\int_{\mathbb{R}^{d}}\left(\int_{s}^{\infty}|(-\Delta)^{\frac{\gamma}{4}}\mathcal{T}_{\psi}(t,s)f(x)|^{2}dt\right)^{\frac{p}{2}}dx
\leq C_{3}\int_{\mathbb{R}^{d}}|f(x)|^{p}dx.
\end{equation}
 \end{itemize}
\end{corollary}

\begin{proof}
Take $\psi_{1}(t,\xi)=-|\xi|^{\frac{\gamma}{q}}$
and $\psi_{2}(t,\xi)=\psi(t,\xi)$ in Theorem \ref{thm: Lp boundedness of S lambda infty}.
Then the symbol $\psi_{1}$ satisfies Assumptions \textbf{(S1)} and \textbf{(S2)}
with $\gamma_{1}=\frac{\gamma}{q}$ and $\kappa_{1}=1$ (see Example \ref{ex: norm xi power gamma}).
In this case, $L_{\psi_{1}}(l)=(-\Delta)^{\frac{\gamma}{2q}}$,
and $\psi_{1}$ is a constant with respect to the variable $t$,
and satisfies that for any $\lambda>0$,
\[
\psi_{1}(t,\lambda \xi)=-|\lambda \xi|^{\frac{\gamma}{q}}
=-\lambda^{\frac{\gamma}{q}}|\xi|^{\frac{\gamma}{q}}=\lambda^{\frac{\gamma}{q}}\psi_{1}(t,\xi).
\]
Therefore, the proof is immediate from Theorem \ref{thm: Lp boundedness of S lambda infty}.
\end{proof}

\begin{corollary}\label{cor: LP for evolution}
Let $q\geq 2$ and let $0<a<\infty$.
Let $\psi\in\mathfrak{S}$ such that $N> d+1+\lfloor\gamma\rfloor$, where $N$ is in \textbf{(S2)}.
Then for any $1<p<\infty$, it holds that
\begin{itemize}
  \item [\rm{(i)}] there exists a constant $C_{1}>0$ depending on $a,d,q,p,\kappa,\mu$ and $\gamma$ such that
  for any $s\geq 0$ and $f\in L^{p}(\mathbb{R}^{d})$,
\begin{equation}\label{eqn: esti of evolution system 1}
\int_{\mathbb{R}^{d}}\left(\int_{s}^{s+a}(t-s)^{q-1}\left|\frac{\partial}{\partial t}\mathcal{T}_{\psi}(t,s)f(x)\right|^{q}dt\right)^{\frac{p}{q}}dx
\leq C_{1}\int_{\mathbb{R}^{d}}|f(x)|^{p}dx
\end{equation}
and
\begin{equation}\label{eqn: esti of evolution system 2}
\int_{\mathbb{R}^{d}}\left(\int_{s}^{s+a}(t-s)^{q-1}\left|\frac{\partial}{\partial s}\mathcal{T}_{\psi}(t,s)f(x)\right|^{q}dt\right)^{\frac{p}{q}}dx
\leq C_{1}\int_{\mathbb{R}^{d}}|f(x)|^{p}dx,
\end{equation}
  \item [\rm{(ii)}]
  if $k\in\mathbb{N}$ and $\psi(r,\xi)=\psi(\xi)$, i.e.,
  $\psi(r,\xi)$ is constant with respect to the variables $r$,
  and $\psi$ satisfies the homogeneity given in \eqref{eqn:homogenioty of psi},
  then there exists a constant $C_{2}>0$ depending on $k,d,q,p,\kappa,\mu$ and $\gamma$ such that
  for any $s\geq 0$ and $f\in L^{p}(\mathbb{R}^{d})$,
\begin{equation}\label{eqn: esti of evolution system 1-2}
\int_{\mathbb{R}^{d}}\left(\int_{s}^{\infty}(t-s)^{kq-1}\left|\frac{\partial^{k}}{\partial t^{k}}\mathcal{T}_{\psi}(t,s)f(x)\right|^{q}dt\right)^{\frac{p}{q}}dx
\leq C_{2}\int_{\mathbb{R}^{d}}|f(x)|^{p}dx
\end{equation}
and
\begin{equation}\label{eqn: esti of evolution system 2-2}
\int_{\mathbb{R}^{d}}\left(\int_{s}^{\infty}(t-s)^{kq-1}\left|\frac{\partial^{k}}{\partial s^{k}}\mathcal{T}_{\psi}(t,s)f(x)\right|^{q}dt\right)^{\frac{p}{q}}dx
\leq C_{2}\int_{\mathbb{R}^{d}}|f(x)|^{p}dx,
\end{equation}
  \item [\rm{(iii)}]
  if $q=2$, there exists a constant $C_{3}>0$ depending on $d,q,p,\kappa,\mu$ and $\gamma$ such that
  for any $s\geq 0$ and $f\in L^{p}(\mathbb{R}^{d})$,
\begin{equation}\label{eqn: esti of evolution system 1-1}
\int_{\mathbb{R}^{d}}\left(\int_{s}^{\infty}(t-s)\left|\frac{\partial}{\partial t}\mathcal{T}_{\psi}(t,s)f(x)\right|^{2}dt\right)^{\frac{p}{2}}dx
\leq C_{3}\int_{\mathbb{R}^{d}}|f(x)|^{p}dx
\end{equation}
and
\begin{equation}\label{eqn: esti of evolution system 2-1}
\int_{\mathbb{R}^{d}}\left(\int_{s}^{\infty}(t-s)\left|\frac{\partial}{\partial s}\mathcal{T}_{\psi}(t,s)f(x)\right|^{2}dt\right)^{\frac{p}{2}}dx
\leq C_{3}\int_{\mathbb{R}^{d}}|f(x)|^{p}dx.
\end{equation}
\end{itemize}
\end{corollary}

\begin{proof}
(i) \enspace Note that
\begin{align}
\frac{\partial}{\partial t}\mathcal{T}_{\psi}(t,s)f(x)&=L_{\psi}(t)\mathcal{T}_{\psi}(t,s)f(x),\label{eqn: dt evol. sys}\\
\frac{\partial}{\partial s}\mathcal{T}_{\psi}(t,s)f(x)&=-L_{\psi}(s)\mathcal{T}_{\psi}(t,s)f(x). \label{eqn: ds evol. sys}
\end{align}
Hence, by \eqref{eqn: dt evol. sys} and \eqref{eqn: LP ineq for q} with $\psi_{1}=\psi_{2}=\psi$
and $l=t$, we get \eqref{eqn: esti of evolution system 1}.
Also, by \eqref{eqn: ds evol. sys} and \eqref{eqn: LP ineq for q} with $\psi_{1}=\psi_{2}=\psi$
and $l=s$, we get \eqref{eqn: esti of evolution system 2}

(ii)\enspace
Since $\psi(t,\xi)=\psi(\xi)$, we have
\begin{align}
\frac{\partial^{k}}{\partial t^{k}}\mathcal{T}_{\psi}(t,s)f(x)
&=L_{\psi}^{k}\mathcal{T}_{\psi}(t,s)f(x)
=L_{\psi^{k}}\mathcal{T}_{\psi}(t,s)f(x),\label{eqn: dtk evol. sys}\\
\frac{\partial^{k}}{\partial s^{k}}\mathcal{T}_{\psi}(t,s)f(x)
&=(-1)^{k}L_{\psi}^{k}\mathcal{T}_{\psi}(t,s)f(x)
=(-1)^{k}L_{\psi^{k}}\mathcal{T}_{\psi}(t,s)f(x), \label{eqn: dsk evol. sys}
\end{align}
by \eqref{eqn: dtk evol. sys} and \eqref{eqn: LP ineq for q} with $\psi_{1}=\psi^{k}$, $\psi_{2}=\psi$ and $l=t$,
we get \eqref{eqn: esti of evolution system 1-2}.
Also, by \eqref{eqn: dsk evol. sys} and \eqref{eqn: LP ineq for homogeneous}
with $\psi_{1}=\psi^{k}$, $\psi_{2}=\psi$ and $l=s$, we get \eqref{eqn: esti of evolution system 2-2}.

(iii)\enspace
By \eqref{eqn: dt evol. sys} and \eqref{eqn: LP ineq q=2} with $\psi_{1}=\psi_{2}=\psi$ and $l=t$,
we get \eqref{eqn: esti of evolution system 1-1}. Also, by \eqref{eqn: ds evol. sys} and \eqref{eqn: LP ineq q=2} with $\psi_{1}=\psi_{2}=\psi$ and $l=s$, we get \eqref{eqn: esti of evolution system 2-1}.
\end{proof}

\begin{corollary}\label{cor: classical LP}
Let $\{P_{t}\}_{t\geq 0}$ be the Poisson semigroup given as in \eqref{eqn: Poisson semigp and kernel} and let $s\geq 0$.
Then for any $q\geq 2$, $1<p<\infty$ and $k\in\mathbb{N}$, there exists a constant $C>0$ depending on $d,q,p,$ and $k$
such that
\begin{equation}\label{eqn: cor-stein's result}
\int_{\mathbb{R}^{d}}\left(\int_{s}^{\infty}(t-s)^{kq-1}\left|\frac{\partial^{k}}{\partial t^{k}}P_{t-s}f(x)\right|^{q}dt\right)^{\frac{p}{q}}dx
\leq C\int_{\mathbb{R}^{d}}|f(x)|^{p}dx
\end{equation}
for any $f\in L^{p}(\mathbb{R}^{d})$.
\end{corollary}

\begin{proof}
It holds that
\[
P_{t}f(x)=\mathcal{F}^{-1}\left(e^{-t|\xi|}\mathcal{F}f(\xi)\right)(x).
\]
Let $\psi(t,\xi)=-|\xi|$. Then $\psi$ satisfies
the Assumptions \textbf{(S1)} and \textbf{(S2)} with $\gamma_{1}=\gamma_{2}=1$ and $\kappa_{1}=\kappa_{2}=1$
(see Example \ref{ex: norm xi power gamma}). Also, we have
\[
\mathcal{T}_{\psi}(t,s)f(x)=\mathcal{F}^{-1}\left(e^{-(t-s)|\xi|}\mathcal{F}f(\xi)\right)(x)=P_{t-s}f(x).
\]
Since
\[
\psi(t,\lambda\xi)=\lambda\psi(t,\xi)
\]
for any $\lambda>0$, by \eqref{eqn: esti of evolution system 1-2}, we obtain \eqref{eqn: cor-stein's result}.
\end{proof}

\begin{remark}
\upshape
The inequality given in \eqref{eqn: cor-stein's result} holds for the heat semigroup.
Especially, \eqref{eqn: cor-stein's result} with $q=2$ and $s=0$
coincides with the result obtained in \cite{Stein 1970-2} (see also \cite{Stein 1970}).
More precisely, in \cite{Stein 1970-2}, the author established \eqref{eqn: cor-stein's result} with $q=2$ and $s=0$
for a symmetric diffusion semigroup.
\end{remark}

\medskip
\noindent
{\bfseries Acknowledgements}\enspace
This paper was supported by a Basic Science Research Program through the NRF funded by the MEST (NRF-2022R1F1A1067601)
and the MSIT (Ministry of Science and ICT), Korea, under the ITRC (Information Technology Research Center) support program (IITP-RS-2024-00437284) supervised by the IITP (Institute for Information \& Communications Technology Planning \& Evaluation).

\end{document}